\documentclass[12pt]{article}
\usepackage{color}
\usepackage{hyperref}
\usepackage{calligra}
\usepackage[T1]{fontenc}
\usepackage[splitrule]{footmisc}
\definecolor{ballblue}{rgb}{0.0, 0.75, 1.0}
\definecolor{green-1}{rgb}{0.0, 0.63, 0.07}
\usepackage[utf8]{inputenc}                      
\usepackage{amsmath,amsfonts}
\usepackage{blkarray}
\usepackage{stmaryrd}
\usepackage{graphicx}
\usepackage{caption}
\usepackage{floatrow}
\usepackage{marginnote}
\usepackage[marginparwidth=2.5cm, marginparsep=0.25cm]{geometry}
\usepackage{listings}
\definecolor{mygray}{rgb}{0.9,0.9,0.9}
\hypersetup{
	pdftex,
	letterpaper=true,
	colorlinks=true,
	pdfpagemode=none,
	urlcolor=blue,
	linkcolor=red,
	citecolor=green-1,
	pdfstartview=FitH,
    pdfborder = {0 0 0},
}

\usepackage{lipsum}

\usepackage[printwatermark]{xwatermark}
\usepackage{tikz}

\makeatletter
\def\@cite#1#2{[\textbf{#1\if@tempswa , #2\fi}]}
\makeatother

\newcommand{\footreferences}{}
\newcommand{\emo}{}

\usepackage[nottoc,notlot,notlof,numbib]{tocbibind}

\usepackage{fnbreak} 
\title{\bf {A digression on Hermite polynomials}}
\author{\small KEITH Y. PATARROYO \vspace{-1ex}\\ \small
\texttt{\href{mailto:kypatarroyot@unal.edu.co}{kypatarroyot@unal.edu.co}}/\texttt{\href{mailto:kypatarroyot@unal.edu.co}{keith.patarroyo@umontreal.ca}}}
\date{\small \today}
\usepackage{abstract}

\usepackage{tocloft}

\newcounter{exo}[section]

\newenvironment{question*}
    {\begin{center}
    \begin{tabular}{p{3cm}p{10cm}}
    \textbf{Question:} & \itshape
    }
    { 
    \end{tabular} 
    \end{center}
    }
\newenvironment{question**}
    {\begin{center}
    \begin{tabular}{p{3cm}p{10cm}}
    \textbf{Question*:} & \itshape
    }
    { 
    \end{tabular} 
    \end{center}
    }

\newif\ifblog
\newif\iftex
\newif\ifkeith
\blogfalse
\keithfalse
\textrue

\usepackage{ulem}

\def\emph#1{\textit{#1}}

\def\image#1#2#3{\begin{center}\includegraphics[#1pt]{#3}\end{center}}

\newtheorem{exercise}{Exercise}
\newtheorem{note}{Note}
\newtheorem{theorem}{Theorem}
\newtheorem{lemma}[theorem]{Lemma}
\newtheorem{definition}[theorem]{Definition}

\newtheorem{example}{Example}

\newenvironment{proof}{\noindent {\sc Proof:}}{ $\bigbox$ \medskip}

\begin{document}
\iftex
\pdfbookmark[1]{Titlepage}{title}
\maketitle
\label{page1}
\pdfbookmark[1]{Abstract}{abstract}
\begin{abstract}
\footnotesize
Orthogonal polynomials are of fundamental importance in many fields of mathematics and science, therefore the study of a particular family is always relevant. In this manuscript, we present a survey of some general results of the Hermite polynomials and show a few of their applications in the connection problem of polynomials, probability theory and the combinatorics of a simple graph. Most of the content presented here is well known, except for a few sections where we add our own work to the subject, nevertheless, the text is meant to be a self-contained personal exposition.
\end{abstract}
\pdfbookmark[1]{\contentsname}{contents}
{
  \hypersetup{linkcolor=black}
  \tableofcontents
}
\fi
\ifblog
Recently a friend of mine made a short document with regards to the \href{https://en.wikipedia.org/wiki/Quantum_harmonic_oscillator}{Hermite polynomials} and the \href{https://en.wikipedia.org/wiki/Hermite_polynomials}{quantum harmonic oscillator}. He suggested me to \href{https://goo.gl/3fY9Jc}{publish it} on my website(Spanish)\cite{vlado}, this motivated me to study a little about the Hermite polynomials and then make my own text about them. 

Yesterday I uploaded to the arXiv my manuscript “A digression on Hermite polynomials“. Since I have mostly seen these polynomials applied to \href{https://en.wikipedia.org/wiki/Quantum_mechanics}{quantum mechanics,} I decided to emphasize the document and this post on other areas of application with a small bias on its  probabilistic usage. There and here, I  survey some general properties of the polynomials and in the end, some applications to the theory of polynomials, probability, and combinatorics are shown. Most of the content is well-known, except for a few sections where I added my own work to the subject, nevertheless, everything is meant to be self-contained.
\fi

\section{A general overview of the polynomials}

We start with the definition of the  polynomials and some details regarding the notation. Afterward, we pursue  a construction and an  explicit expression for them.

\subsection{Definition}\label{sesu:he-def-gen-fun}

The \emph{Chebyshev}\footnote{We are calling the functions $\{He_n(x)\}_{n=0}^{\infty}$, \emph{Chebyshev-Hermite polynomials} following the \iftex convention \fi  \ifblog\href{http://www.oxfordreference.com/view/10.1093/oi/authority.20110803095604775}{convention}\fi of \cite{kendall} and \cite{blinnikov} to differentiate them from $\{H_n(x)\}_{n=0}^{\infty}$ (see Note \ref{note:herm-notation} in the text). However they are most commonly known by \emph{Hermite polynomials} alone(hence the \iftex\hyperref[page1]{\bf{title}} \fi\ifblog <a href="#post-768">title</a>\fi of this article), \iftex \linebreak \fi e.g. see \cite{nist-sp-fun},\cite{magnus},\cite{johnson}. It is our hope that \emph{"Chebyshev-Hermite polynomials"} becomes generally accepted since at the present time it is only known by some people in the statistical community \cite{kotz}.  Also, note that these polynomials are NOT the Hermite-Chebyshev polynomials defined in \cite{batahan}.\emo}-\iftex Hermite polynomials \fi  \ifblog\href{https://en.wikipedia.org/wiki/Hermite_polynomials}{Hermite polynomials}\fi $\{He_n(x)\}_{n=0}^{\infty}$ arise naturally when we take the ratio of two \iftex standard Gaussian distributions\cite{johnson}\fi \ifblog\href{https://en.wikipedia.org/wiki/Normal_distribution#Standard_normal_distribution}{standard Gaussian distributions}\fi , this motivates the following definition. 

\begin{definition} \label{def:herm-prob-gen}
The Chebyshev-Hermite polynomials $\{He_n(x)\}_{n=0}^{\infty}$  are defined in the interval $-\infty < x < \infty$ and they are generated by the ratio of a displaced standard Gaussian distribution and a standard Gaussian distribution centered in zero,

\begin{equation}  \label{eq:herm-prob-gen}
\frac{\phi(x-t)}{\phi(x)}= e^{xt-\frac{t^2}{2}} = \sum_{n=0}^\infty He_n(x)\frac{t^n}{n!} =G(x,t).
\end{equation}

\end{definition}

This definition tells us how to compute the polynomials and it shows that the Chebyshev-Hermite polynomials have to some extent a connection with the Gaussian distribution. Furthermore we can see right away that this  \iftex \emph{generating function} \fi \ifblog \href{https://en.wikipedia.org/wiki/Generating_function#Ordinary_generating_function_(OGF)}{generating function} \fi \eqref{eq:herm-prob-gen} is quite similar to the moment generating function $m_Y(t)= e^{xt+\frac{t^2}{2}}$ of a random variable $Y\stackrel{d}{=}\mathcal{N}(x,1)$ with normal distribution. A substantial relation of $ He_n(x)$ with $Y$ and ${\bf E}[Y^n](x)$ will arise when we consider the \emph{integral representation} and the \emph{connection problem} of the polynomials in Sections \ref{sesu:her-int-rep} and \ref{sesu:her-con-pro} respectively. 

\begin{note}\label{note:herm-notation}
Hermite polynomials are standardized in two different ways depending on their use, they are: the \iftex\underline{\textbf{Chebyshev}-Hermite polynomials,} \fi \ifblog\href{https://en.wikipedia.org/wiki/Hermite_polynomials#Definition}{Chebyshev-Hermite polynomials,}\fi$\{He_n(x)\}_{n=0}^{\infty}$ (regularly applied in probability \cite{johnson}) and the \iftex\underline{Hermite polynomials,} \fi \ifblog\href{https://en.wikipedia.org/wiki/Hermite_polynomials#Definition}{Hermite polynomials,}\fi $\{H_n(x)\}_{n=0}^{\infty}$ (frequently applied in physics \cite{vlado},\cite{hassani}). Sometimes the naming of the polynomials change from author to author, however the mathematical notation for both is currently well distinct in the technical literature \cite{nist-sp-fun}. They both share the following relationship,

\begin{equation} \label{eq:hermite-conv}
H_n(x) = 2^{\frac{n}{2}} He_n(\sqrt{2}x).
\end{equation}
\iftex
\vspace{8pt}
\fi
\end{note}

\subsection{Construction by a Gram-Schmidt algorithm}
\label{sesu:he-def-gram-schmidt}

One way to naturally obtain \textit{classical}\footnote{ The Chebyshev-Hermite polynomials are  a  \iftex \textit{classical} orthogonal polynomial \fi \ifblog\href{https://en.wikipedia.org/wiki/Classical_orthogonal_polynomials}{classical orthogonal polynomial}\fi  sequence, Why classical? Are there any other? e.g. \textit{quantum}?, see \cite{hassani} or \cite{ismail} for a precise definition.\emo} \iftex orthogonal polynomials \fi \ifblog\href{http://www.damtp.cam.ac.uk/user/na/PartIB/Lect03.pdf}{orthogonal polynomials}\fi is by constructing a set of polynomials $\{p_n(x)\}_{n=0}^{\infty}$ with $\text{deg}[p_n(x)] = n$ such that they are \iftex orthogonal\cite{nist-sp-fun} \fi\ifblog \href{https://en.wikipedia.org/wiki/Orthogonality#Orthogonal_functions}{orthogonal}\fi under an inner product with a weight function $w(x)$  on the interval $(a, b)$.  The following theorem guarantees that such set of polynomials exist and are unique, moreover it tell us a way to iteratively compute them.

\begin{theorem}\label{th:gram-schmidt}
Given an  inner product with a weight function $w(x)$ on an interval $(a, b)$, there exists a sequence of  polynomials $\{p_n(x)\}_{n=0}^{\infty}$ orthogonal with respect to the inner product. This sequence is uniquely determined up to constant factors\footnote{As we will see later the Chebyshev-Hermite polynomials are  \iftex monic\cite{nist-sp-fun}\fi \ifblog\href{https://en.wikipedia.org/wiki/Monic_polynomial}{monic}\fi, Exercise \ref{ex:her-prop}, and according to Theorem \ref {th:gram-schmidt} they are only undetermined up to a constant factor.  Now lets consider $2^{-n}H_n(x) = 2^{\frac{-n}{2}} He_n(\sqrt{2}x)$, these polynomials are monic too !! What is wrong? Is the weight function for $\{H_n(x)\}_{n=0}^{\infty}$ equal to the weight function for $\{He_n(x)\}_{n=0}^{\infty}$?\emo}.
\end{theorem}

\begin{proof}
See \cite{Sueli} for complete details on the proof. Here we give an outline on the proof, this requires constructing the polynomials using a \iftex Gram-Schmidt \fi \ifblog\href{https://en.wikipedia.org/wiki/Gram\%E2\%80\%93Schmidt_process}{Gram-Schmidt}\fi  like algorithm to orthogonalize the linear independent  sequence $\{1,x,x^2,...\}$ using a determined inner product with a weight function $w(x)$. This construction is not the most efficient approach to build the polynomials, since it requires lots of computations, see \cite{Sueli} for an \iftex example \fi \ifblog\href{https://homepage.tudelft.nl/11r49/documents/wi4006/orthopoly.pdf}{example}\fi  of this procedure, later we will see better constructions algorithms. 
\end{proof}

If we choose the weight function to be an unnormalized standard Gaussian distribution, $w(x)=e^{\frac{-x^2}{2}}$\iftex\footnotemark\fi \ifblog \footnote{Why choose $e^{\frac{-x^2}{2}}$ as a weight function anyway? For an arbitrary function to be a weight function, it must be positive, continous and all its moments $\mu_n :=\int_a^b x^n w(x)dx$ should exist \cite{spe-func}. In the other hand the classification of \textit{classic} orthogonal polynomials \cite{hassani}  yields that a function of the form $e^{\frac{-x^2}{2}}$ \emph{must generate} a sequence of classic orthogonal polynomials.\emo}\fi  \ \ and the interval $(-\infty,\infty)$, we obtain the sequence of Chebyshev-Hermite polynomials as the set of orthogonal polynomials, the orthogonality is proven later in Lemma \ref{le:her-prob-orth}. Hence if we are working in an application where the inner product has the alleged weight function, the natural way to write an arbitrary function, $f \in L^2_w(\mathbb{R})$ as a series of orthogonal functions, is using the Chebyshev-Hermite polynomials. We can do this since the Chebyshev-Hermite polynomials are a Complete Orthogonal System  in $L^2_w(\mathbb{R})$\cite{hilbert}. 

\subsection{Explicit expression} \label{sesu:her-exp-expre}

A remarkable, however vague result is a general \iftex \textit{ explicit expression }\fi \ifblog \href{https://en.wikipedia.org/wiki/Hermite_polynomials#Explicit_expression}{explicit expression} \fi for the Chebyshev-Hermite polynomials. To obtain it we start with the generating function \eqref{eq:herm-prob-gen} and expand the exponential argument,

$$e^{xt-\frac{t^2}{2}}=\sum_{k=0}^{\infty}\frac{\left(xt-t^2/2\right)^k}{k!} = \sum_{k=0}^{\infty}\sum_{j=0}^{k}{{k}\choose{j}}\frac{\left(xt\right)^{k-j} (-1)^{j}\left(t\right)^{2j}}{2^j k!},$$

where we used the binomial theorem to obtain the last expression. Now we redefine indexes, so we let $n=k+j$ and rewrite both sums,

\iftex \footnotetext{Why choose $e^{\frac{-x^2}{2}}$ as a weight function anyway? For an arbitrary function to be a weight function, it must be positive, continous and all its \iftex moments \fi \ifblog\href{https://homepage.tudelft.nl/11r49/documents/wi4006/orthopoly.pdf}{moments}\fi  $\mu_n :=\int_a^b x^n w(x)dx$ should exist \cite{spe-func}. In the other hand the classification of \textit{classic} orthogonal polynomials \cite{hassani}  yields that a function of the form $e^{\frac{-x^2}{2}}$ \emph{must generate} a sequence of classic orthogonal polynomials.\emo}\fi

$$e^{xt-\frac{t^2}{2}}=\sum_{n=0}^{\infty}\sum_{j=0}^{\left\lfloor{\frac{n}{2}}\right\rfloor}{{n-j}\choose{j}}\frac{\left(x\right)^{n-2j} (-1)^{j}\left(t\right)^{n}}{2^j (n-j)!}=\sum_{n=0}^{\infty}\frac{t^n}{n!}\sum_{j=0}^{\left\lfloor{\frac{n}{2}}\right\rfloor}\frac{n!\left(x\right)^{n-2j} (-1)^{j}}{2^j (n-2j)!j!},$$

where $\left\lfloor{.}\right\rfloor$ is the floor function. Now comparing with \eqref{eq:herm-prob-gen} we find that\footnote{ An explicit expression  is wonderful for calculating the polynomials in a computer, however, it is cumbersome (yet \eqref{eq:her-phy-expl} and \eqref{eq:her-prob-expl} will be really important in some applications, Section \ref{se:her-app}), it is hard to see some  properties of the polynomials with the explicit expression. As we will see later other \textit{generating methods}, Section \ref{se:he-misc} (recurrence relation, Rodrigues formula, ODE, integral representation) may highlight the importance of a property more than another. This is maybe why orthogonal polynomials  are so powerful, one is able to apply  different perspectives  to solve a  problem.\emo},

\begin{equation} \label{eq:her-prob-expl}
He_n(x) = n! \sum_{j=0}^{\left\lfloor{\frac{n}{2}}\right\rfloor}\frac{(-1)^{j} x^{n-2j} }{2^j (n-2j)!j!}.
\end{equation}

We will later need in this document the explicit expression for the \emph{Hermite polynomials}, $H_n(x)$, so using equation \eqref{eq:hermite-conv} we can write their \iftex explicit expression, \fi \ifblog \href{https://en.wikipedia.org/wiki/Hermite_polynomials#Explicit_expression}{explicit expression}, \fi 

\begin{equation} \label{eq:her-phy-expl}
H_n(x) = n! \sum_{j=0}^{\left\lfloor{\frac{n}{2}}\right\rfloor}\frac{(-1)^{j} (2x)^{n-2j} }{(n-2j)!j!}.
\end{equation}
\begin{exercise}\label{ex:her-prop}
Check the following special values for the Chebyshev-Hermite \iftex \break\fi polynomials:
\begin{itemize}
\item Leading Coefficient: $He^{(n)}_n /n! = 1$, then the polynomials are \iftex monic\cite{nist-sp-fun}\fi \ifblog\href{https://en.wikipedia.org/wiki/Monic_polynomial}{monic}\fi.
\item Final Coefficient: $He_{2n}(0)=\frac{(-1)^{n}(2n)!}{n!2^n}$ and $He_{2n+1}(0)=0$.
\item Parity : $He_n(-x)=(-1)^{n}He_n(x)$.
\end{itemize}
\end{exercise}

The reader is encouraged to prove some of these properties with the other generating methods for the polynomials, Sections \ref{sesu:he-def-gen-fun}, \ref{sesu:he-def-gram-schmidt} and Section \ref{se:he-misc}. This will make him more comfortable with the different representations of the polynomials.

With the explicit expression \eqref{eq:her-prob-expl} we can make a table for the first 
six Chebyshev-Hermite polynomials and a plot of them. We present this data in the following figure, 

\begin{figure}[h!]
\image{width = 420}{https://keithpatarroyo.files.wordpress.com/2018/09/hermite_poly_table1.png}{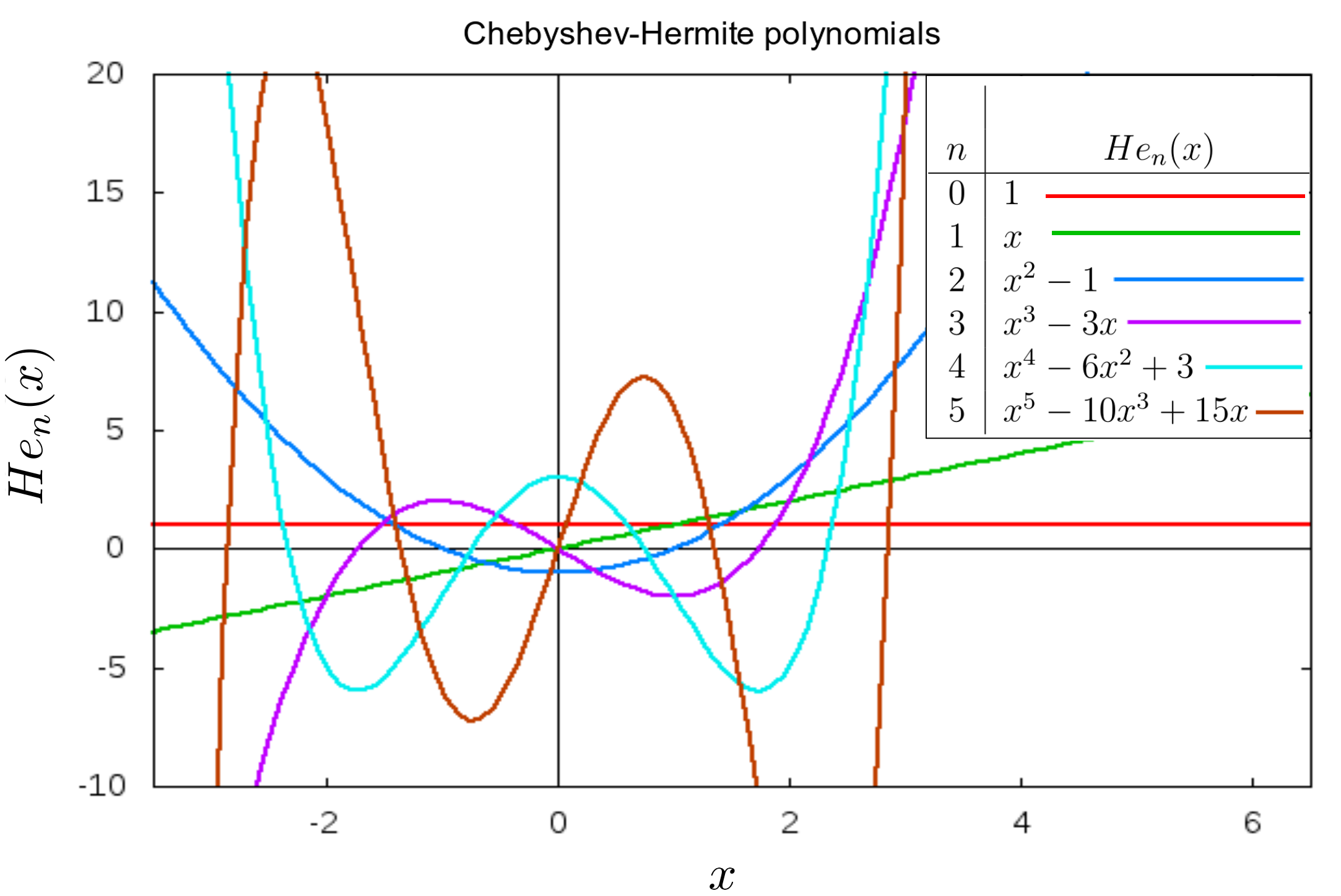}
\captionof{figure}{The first six Chebyshev-Hermite polynomials $He_n(x)$. \emo}\label{fig:hermite}
\end{figure}
 
We can rapidly check the properties of Exercise \ref{ex:her-prop} with the table in Figure \ref{fig:hermite}. Also, from the figure we see other beautiful properties of the polynomials, the polynomial of order $n$ has $n$ real roots and between any two roots of a polynomial there is a root of a higher order polynomial, these facts will be of use later.

\section{Some fundamental results}\label{se:he-misc}

We proceed next to present some features of the Chebyshev-Hermite polynomials that are commonly shown in the treatment of orthogonal polynomials.

\subsection{Recurrence relation}

All classical orthogonal polynomials follow a  \iftex recurrence relation as is well known in the technical literature \cite{nist-sp-fun}, in some treatments of orthogonal polynomials the generating function is derived from the recurrence relations \cite{hassani}. \fi  \ifblog \href{http://mathworld.wolfram.com/RecurrenceEquation.html}{recurrence relation}, in my previous post \href{https://keithpatarroyo.wordpress.com/2016/07/12/relation-between-legendre-polynomial-generating-function-and-legendre-differential-equation/}{Relation between Legendre polynomial Gen. function and Legendre diff. eq.} we derived the generating function from the recurrence relation.\fi We do the opposite process and derive two recurrence relations from the generating function \eqref{eq:herm-prob-gen}.

\begin{lemma}\label{le:her-recurrence} 
The sequence of the Chebyshev-Hermite polynomials, $\{He_n(x)\}_{n=0}^{\infty}$,  satisfy the two following recurrence relations,

\iftex \begingroup\makeatletter\def\f@size{11.5}\check@mathfonts
\fi
\begin{equation}\label{eq:he-re-1}
He'_{n}(x) = nHe_{n-1}(x), \qquad \textnormal{ for } \: n\geq 1,
\end{equation}
\begin{equation}\label{eq:he-re-2}
He_{n+1}(x)-xHe_{n}(x)+nHe_{n-1}(x) = 0, \qquad \textnormal{ for } \: n\geq 1.
\end{equation}\iftex\endgroup\fi

\end{lemma}

\begin{proof}
We start with the generating function \eqref{eq:herm-prob-gen}, $G(x,t) = e^{xt-\frac{t^2}{2}}$ and we note that it satisfies the following two differential equations,
\begin{eqnarray*}
\frac{\partial G}{\partial x}(x,t) &=& te^{xt-\frac{t^2}{2}} = t G(x,t),  \\ 
\frac{\partial G}{\partial t}(x,t) &=& xe^{xt-\frac{t^2}{2}}-te^{xt-\frac{t^2}{2}} = (x-t) G(x,t).
\end{eqnarray*}

Now if we plug the series expansion for the generating function \eqref{eq:herm-prob-gen} in the previous differential equations, we can obtain equalities relating the coefficients of the series. First for the derivative with respect to $x$ we have,  

$$\frac{\partial G}{\partial x}(x,t) =\sum_{n=1}^\infty He'_n(x)\frac{t^n}{n!}= \sum_{m=1}^\infty m He_{m-1}(x)\frac{t^{m}}{m!} = \sum_{n=0}^\infty He_n(x)\frac{t^{n+1}}{n!} = t G(x,t),$$

where we made $m=n+1$ in the third equality, now equating for the coefficients of the series we find \eqref{eq:he-re-1}. Finally the derivative with respect to $t$ yields,  

\iftex
\begingroup\makeatletter\def\f@size{11.5}\check@mathfonts
\fi
\begin{eqnarray*}
\frac{\partial G}{\partial t}(x,t) =\sum_{n=1}^\infty nHe_n(x)\frac{t^{n-1}}{n!}=\sum_{k=0}^\infty He_{k+1}(x)\frac{t^k}{k!} &=& \sum_{n=0}^\infty xHe_n(x)\frac{t^{n}}{n!} -\sum_{m=0}^\infty m He_{m-1}(x)\frac{t^{m}}{m!},\\
&=& \sum_{n=0}^\infty xHe_n(x)\frac{t^{n}}{n!} - \frac{\partial G}{\partial x}(x,t),\\
\\
&=& (x-t) G(x,t),
\end{eqnarray*}
\iftex
\endgroup
\fi

where we took the series expansion of $\frac{\partial G}{\partial x}(x,t)$ from the previous manipulation. Now equating for the coefficients of the series we find \eqref{eq:he-re-2}. This proves the lemma. \end{proof}

Now we present an example showing the power of this recurrence relation. We will show later another solution for this example in Section \ref{sesu:her-con-pro}.

\begin{example}\label{exa:her-convolution}
We compute the \iftex convolution \fi \ifblog\href{https://en.wikipedia.org/wiki/Convolution}{convolution}\fi of a standard Gaussian  distribution $\phi(x)$ with the nth Chebyshev-Hermite polynomial $He_n(x)$, also called the rescaled \iftex Weierstrass-Gauss transform \fi \ifblog\href{https://en.wikipedia.org/wiki/Weierstrass_transform}{Weierstrass-Gauss transform}\fi \cite{zayed},\cite{hirschman} of the Chebyshev-Hermite polynomial $\mathcal{W}\left[2^{\frac{n}{2}}He_n\left(\frac{x}{\sqrt{2}}\right)\right]$, 

$$
\mathcal{W}\left[2^{\frac{n}{2}}He_n\left(\frac{x}{\sqrt{2}}\right)\right]=( \phi* He_n )(x) = \int_{-\infty}^{\infty} (2\pi)^{-1/2}e^{-\frac{(y-x)^2}{2}}He_{n}(y)dy.
$$

We claim that $( \phi* He_n)(x) = x^n$ ,  we see that the convolution yields the moments ${\bf E}[He_n(Y)]$ of  a random variable $Y\stackrel{d}{=}\mathcal{N}(x,1)$ with normal distribution. Then we clearly notice from the table in Figure \ref{fig:hermite} that  for $n=0,1$,

$$( \phi* He_0)(x) = {\bf E}[Y^{0}] = 1, \qquad ( \phi* He_1)(x) = {\bf E}[Y^{1}] = x.$$

Now we employ induction on $n$, suppose then for $n-1$ we have, 
\begin{equation}\label{eq:exp-1-ind-hyp}
( \phi* He_{n-1} )(x) = \int_{-\infty}^{\infty} (2\pi)^{-1/2}e^{-\frac{(y-x)^2}{2}}He_{n-1}(y)dy=x^{n-1}.
\end{equation}

Consider now $( \phi* He_n)(x)$ and lets use both \eqref{eq:he-re-1},\eqref{eq:he-re-2} with $n\rightarrow n-1$,

\begin{equation} \label{eq:ex-1-1}
\int_{-\infty}^{\infty} (2\pi)^{-1/2}e^{-\frac{(y-x)^2}{2}}He_{n}(y)dy = 
\int_{-\infty}^{\infty} (2\pi)^{-1/2}e^{-\frac{(y-x)^2}{2}}(yHe_{n-1}(y)-He'_{n-1}(y))dy. 	
\end{equation}

We can then compute the integral with the term $He'_{n-1}$ by integration by parts, 

$$\int_{-\infty}^{\infty} (2\pi)^{-1/2} e^{-\frac{(y-x)^2}{2}}He'_{n-1}(y)dy =  \int_{-\infty}^{\infty} (2\pi)^{-1/2} (y-x)e^{-\frac{(y-x)^2}{2}}He_{n-1}(y)dy,
$$

where we took into account that $e^{-\frac{(y-x)^2}{2}}$ multiplied by any polynomial vanish for infinite $x$. Now replacing the last expression in \eqref{eq:ex-1-1} we obtain,

$$\int_{-\infty}^{\infty} (2\pi)^{-1/2}e^{-\frac{(y-x)^2}{2}}He_{n}(y)dy =x 
\int_{-\infty}^{\infty} (2\pi)^{-1/2}e^{-\frac{(y-x)^2}{2}}(He_{n-1}(y))dy. $$
Finally using the induction hypothesis \eqref{eq:exp-1-ind-hyp} in the last expression we obtain,
\begin{equation} \label{eq:herm-convolution}
\mathcal{W}\left[2^{\frac{n}{2}}He_n\left(\frac{x}{\sqrt{2}}\right)\right]={\bf E}[He_n(Y)]=( \phi* He_n )(x) =  x^{n},
\end{equation}

which is the desired result.
\end{example}

\begin{exercise}\label{ex:convolution-pr}
Compute the integral \eqref{eq:herm-convolution} starting from the generating function \eqref{eq:herm-prob-gen}, multiply both sides by $(2\pi)^{-1/2}e^{\frac{(x-\hat{x})^2}{2}}$ integrate in $\mathbb{R}$ and compare the terms in the series.
\end{exercise}

\subsection{Rodrigues formula}
\label{sesu:Rodrigues}
Now we derive the so-called \iftex Rodrigues formula \fi \ifblog\href{https://en.wikipedia.org/wiki/Rodrigues\%27_formula}{Rodrigues formula}\fi for the Chebyshev-Hermite polynomials, this formula is extremely useful to solve many problems quickly. We start with our generating function \eqref{eq:herm-prob-gen}, we write it in a convenient way and recognize the Taylor expansion coefficients as,

$$\left[ \frac{\partial^n }{\partial t^n}\left( e^{xt-\frac{t^2}{2}} \right) \right]_{t=0} = e^{\frac{x^2}{2}}\left[ \frac{\partial^n }{\partial t^n}\left( e^{\frac{-(t-x)^2}{2}} \right) \right]_{t=0} = He_n(x).$$

Using the identity $\frac{\partial f(t-x)}{\partial t} = -\frac{\partial f(t-x)}{\partial x}$ in the above expression yield,

$$e^{\frac{x^2}{2}}\left[ \frac{\partial^n }{\partial t^n}\left( e^{\frac{-(t-x)^2}{2}} \right) \right]_{t=0} = e^{\frac{x^2}{2}}(-1)^n\left[ \frac{\partial^n }{\partial x^n}\left( e^{\frac{-(t-x)^2}{2}} \right) \right]_{t=0} = (-1)^n e^{\frac{x^2}{2}} \frac{d^n }{d x^n}\left( e^{\frac{-x^2}{2}} \right).$$

Then the Rodrigues formula for the Chebyshev-Hermite polynomial is,

\begin{equation}\label{eq:her-prob-rodri}
He_n(x) = (-1)^n e^{\frac{x^2}{2}} \frac{d^n }{d x^n}\left( e^{\frac{-x^2}{2}} \right).
\end{equation}
\begin{exercise}\label{ex:her-prop-rodrigues}
Show that the Chebyshev-Hermite polynomials can also be defined in the following two ways,
\begin{equation}\label{eq:ex-3}
He_n(x) =  e^{\frac{x^2}{4}} \left(\frac{x}{2}-\frac{d}{d x} \right)^n \cdot \left[ e^{\frac{-x^2}{4}} \right], \qquad He_n(x) = \left(x-\frac{d }{dx} \right)^n  \cdot 1 .
\end{equation}
\end{exercise}

Now we  prove the \iftex orthogonality \fi\ifblog \href{https://en.wikipedia.org/wiki/Orthogonality#Orthogonal_functions}{orthogonality}\fi of Chebyshev-Hermite polynomials, we decide to present this proof since it shows the power of the Rodrigues formula. Again the reader is encouraged to prove this with the other generating methods for the polynomials.

\begin{lemma} \label{le:her-prob-orth}
The Chebyshev-Hermite polynomials are orthogonal with respect to the weight function $w(x) = e^{\frac{-x^2}{2}}$ in the interval $(-\infty, \infty)$. In other words,
\begin{equation} \label{eq:her-prob-orth} 
\int_{-\infty}^{\infty} e^{\frac{-x^2}{2}}He_n(x)He_{n'}(x)dx=\sqrt{2\pi} n! \delta_{nn'}.
\end{equation}
\end{lemma}

\begin{proof}
We start by expressing $He_n(x)$ with the Rodrigues formula, \eqref{eq:her-prob-rodri}, then the integral yields,
$$\int_{-\infty}^{\infty} e^{\frac{-x^2}{2}}He_n(x)He_{n'}(x)dx=\int_{-\infty}^{\infty} (-1)^n \frac{d^n }{d x^n}\left( e^{\frac{-x^2}{2}} \right) He_{n'}(x)dx.$$

For the case $n \neq n'$ assume without loss of generality that $n>n'$, now we integrate by parts $n$ times the above expression, we take into account  the fact that $e^{\frac{-x^2}{2}}$ and all its derivatives vanish for infinite $x$,

\begin{eqnarray*}
\int_{-\infty}^{\infty} e^{\frac{-x^2}{2}}He_n(x)He_{n'}(x)dx &=& (-1)^{n-1} \int_{-\infty}^{\infty}  \frac{d^{n-1} }{d x^{n-1}}\left( e^{\frac{-x^2}{2}} \right) He_{n'}^{(1)}(x)dx,\\
&=& \qquad \qquad \qquad \vdots\\
&=& (-1)^{n-n'} \int_{-\infty}^{\infty}  \frac{d^{n-n'} }{d x^{n-n'}}\left( e^{\frac{-x^2}{2}} \right) He_{n'}^{(n')}(x)dx,\\
&=& 0,
\end{eqnarray*}

where its clear that the $n'+1$ derivative of a $n'$ degree polynomial is zero.

For the case $n = n'$, we perform the same procedure as above, we  integrate by parts $n$ times ,

\begin{eqnarray*}
\int_{-\infty}^{\infty} e^{\frac{-x^2}{2}}He_n(x)He_{n}(x)dx 
&=& (-1)^{0} \int_{-\infty}^{\infty}  \frac{d^{0} }{d x^{0}}\left( e^{\frac{-x^2}{2}} \right) He_{n}^{(n)}(x)dx,\\
&=& n! \int_{-\infty}^{\infty}   e^{\frac{-x^2}{2}}  dx,\\
&=&  n! \sqrt{2\pi},
\end{eqnarray*}

where we made use of the leading coefficient of the Chebyshev-Hermite polynomials, Exercise \ref{ex:her-prop}, and the Gaussian integral. This proves the lemma. \end{proof}

Now a little exercise where one can apply the Rodrigues formula and the method applied above to prove Lemma \ref{le:her-prob-orth} of successive integration by parts.

\begin{exercise}

Obtain the \iftex \underline{inverse explicit expression} \fi \ifblog\href{https://en.wikipedia.org/wiki/Hermite_polynomials#Inverse_explicit_expression}{inverse explicit expression}\fi from \eqref{eq:her-prob-expl} for the Chebyshev-Hermite polynomials, 

\begin{equation} \label{eq:her-prob-inverse}
x^n = n! \sum_{j=0}^{\left\lfloor{\frac{n}{2}}\right\rfloor} \frac{He_{n-2j}(x)}{2^{j}(n-2j)!j!}.
\end{equation}

Suggestion: Divide in two cases $n=2n$ and $n=2n+1$, then expand the power as a linear combination of Chebyshev-Hermite polynomials. Use orthogonality to find the coefficients,  do the necessary integrals with the process described above. Join the even and odd result to find \eqref{eq:her-prob-inverse}.
\end{exercise}

We will later need in this document the inverse explicit expression for the  \textit{Hermite polynomials}, $H_n(x)$, so using equation \eqref{eq:hermite-conv} and \eqref{eq:her-prob-inverse} we can write their inverse explicit expression from \eqref{eq:her-phy-expl} as,
\begin{equation} \label{eq:her-phys-inverse}
(2x)^n = n! \sum_{j=0}^{\left\lfloor{\frac{n}{2}}\right\rfloor} \frac{H_{n-2j}(x)}{(n-2j)!j!}.
\end{equation}
Next, we apply some of the results we obtained in the previous sections to a very important application of the polynomials, numerical integration.

\subsection{Gauss-Hermite Quadrature}
\label{sesu:quadrature-1d}

Imagine one would like to compute a very difficult integral in the interval $(-\infty,\infty)$, since the interval is unbounded, the direct application of classical techniques (\iftex trapezoid \fi \ifblog\href{https://en.wikipedia.org/wiki/Trapezoidal_rule}{trapezoid}\fi or unmodified \iftex Monte Carlo\fi \ifblog\href{https://en.wikipedia.org/wiki/Monte_Carlo_integration}{Monte Carlo}\fi) for the numerical computation of the integral are unsuccessful. Remarkably the Chebyshev-Hermite polynomials can be used to tackle this problem \textit{directly} using the so-called \iftex numerical quadrature or Gauss quadrature.\fi \ifblog\href{https://en.wikipedia.org/wiki/Gaussian_quadrature}{numerical quadrature or Gauss quadrature.}\fi

The \textit{quadrature} rule  for a general family of orthogonal polynomials $\{p_n(x)\}_{n=0}^{\infty}$ with respect to a weight $w(x)$ on an interval $(a,b)$ is the result of following theorem, 

\begin{theorem}\label{th:quadrature-1d}
Let $f(x)$ be a $2N-1$ degree polynomial and $\{x_i\}_{i=0}^{N}$ be the zeros of $p_N(x)$, then the following quadrature formula holds,
\begin{equation} \label{eq:gauss-quadrature-1d}
\int_a^b w(x) f(x)dx=\sum_{i=1}^{N}w_{i,N} f(x_i),\qquad w_{i,N}=\frac{a_N}{a_{N-1}}\cdot \frac{\int_a^b w(x) [p_{N-1}(x)]^2 dx}{p_N'(x_i)p_{N-1}(x_i)},
\end{equation}
where $a_N$ is the $Nth$ coefficient of the orthogonal polynomial $p_N(x)$.
\end{theorem}

\begin{proof}
See \cite{hildebrand} for \iftex details. \fi \ifblog\href{http://www.damtp.cam.ac.uk/user/na/PartIB/Lect04.pdf}{details.}\fi 
\end{proof}

It can also be proven, see \cite{hildebrand},  that all the zeros of $p_n(x)$ for $n\geq 1$ are real, distinct and lie in 
$(a,b)$, hence \eqref{eq:gauss-quadrature-1d} is always computable. 

If $f(x)$ is not a $2N-1$ polynomial, \eqref{eq:gauss-quadrature-1d} is not exact, roughly the discrepancy is a result of how well $f(x)$ is approximated by a polynomial. However if $f(x)$ is continuous the approximation gets better as we increase the number of quadrature points \cite{spe-func}.

For the particular case of the Chebyshev-Hermite polynomials, using equations \eqref{eq:her-prob-orth}, \eqref{eq:he-re-1} and Exercise \ref{ex:her-prop} we obtain,
\begin{equation} \label{eq:hermite-quadrature-1d}
\int_{-\infty}^{\infty} e^{\frac{-x^2}{2}}  f(x)dx=\sum_{i=1}^{N}w_{i,N} f(x_i),\qquad w_{i,N}=\frac{\sqrt{2 \pi}N!}{[N He_{N-1}(x_i)]^2},
\end{equation}
where  $\{x_i\}_{i=0}^{N}$ are the zeros of $He_N(x)$.

We can use the Chebyshev-Hermite polynomials to solve the original problem of a difficult integral of the form  $I=\int_{-\infty}^{\infty}  f(x)dx$ for arbitrary complicated $f(x)$, one way to do this is write $f(x)$ as $\hat{f}(x) e^{\frac{-x^2}{2}} $(some approaches are described in Section \ref{sesusu:Fourier-Hermite}), then the integral can be approximated as, 

$$ I=\int_{-\infty}^{\infty} e^{\frac{-x^2}{2}}  \hat{f}(x)dx\approx \sum_{i=1}^{N}w_{i,N} \hat{f}(x_i).$$

Even though this method might require two approximations, one can \textit{quickly} improve the estimate by a high order quadrature rule since the procedure is just a matter of evaluate $\hat{f}(x)$ in some points, multiply by $w_{i,N}$ (both the points and weights can be stored in memory) and sum the outcomes.

\subsection{Differential Equation(ODE)}

The Chebyshev-Hermite polynomials also arise in many fields since they satisfy the eigenvalue problem,
\begin{equation}\label{her-prob-diff-eq}
u''-xu' = -n u,
\end{equation}

which for positive integer $n$ is widely regarded as the \iftex \textit{Hermite equation}. \fi \ifblog\href{https://en.wikipedia.org/wiki/Hermite_polynomials#Hermite's_differential_equation}{Hermite equation.}\fi  

Its possible to solve this equation by the power series method and show that one family of solutions is indeed the Chebyshev-Hermite polynomials, \cite{ency}. However here we only show that the Chebyshev-Hermite polynomials satisfy 
\eqref{her-prob-diff-eq}, this is summarized in the following lemma,

\begin{lemma}\label{le:hermite-ode}
The sequence of the Chebyshev-Hermite polynomials, $\{He_n(x)\}_{n=0}^{\infty}$,  satisfy the Hermite differential equation \eqref{her-prob-diff-eq}.
\end{lemma}

\begin{proof}
We start by replacing the recurrence relation \eqref{eq:he-re-1} in the recurrence relation \eqref{eq:he-re-2}, this yields,
$$He_{n+1}(x)-xHe_{n}(x)+He'_{n}(x) = 0.$$

Now taking the derivative of the previous equation we obtain,
$$He'_{n+1}(x)-He_{n}(x)-xHe'_{n}(x)+He''_{n}(x) = 0.$$

Using \eqref{eq:he-re-1} with the term $He'_{n+1}(x)$ in the previous equation, we obtain,

\begin{equation}
He''_{n}(x)-xHe'_{n}(x) = -nHe_n(x).
\end{equation}

This proves the lemma. \end{proof}

Its convenient  to introduce the orthogonal \textbf{Chebyshev-Hermite functions } $he_n(x)$, defined by,

\begin{equation}\label{eq:chebyshev-hermite-functions}
he_n(x) = e^{\frac{-x^2}{4}}He_n(x).
\end{equation}

In the following exercise the reader is encouraged to check some elementary properties of the Chebyshev-Hermite functions,
\begin{exercise}\label{ex:prob-herm-fun-ode-prop}
Show that the Chebyshev-Hermite functions $\{he_n(x)\}_{n=0}^{\infty}$ are orthogonal in $L^2(\mathbb{R})$(with weight function 1) and they satisfy,

\begin{equation}\label{eq:herm-fun-rec}
\frac{x}{2}he_n(x)+he'_n(x) = n he_{n-1}(x),
\end{equation}

\begin{equation}\label{eq:weber}
he''_{n}(x)+\left(-\frac{x^2}{4}+n+\frac{1}{2}\right)he_n(x) = 0.
\end{equation}
\end{exercise}

Equation \eqref{eq:weber} is known as \iftex \textit{Weber equation} \fi \ifblog \href{https://www.encyclopediaofmath.org/index.php/Weber_equation}{Weber equation}\fi \cite{weber}, it comes up in the study of the \iftex Laplace's equation \fi \ifblog \href{https://en.wikipedia.org/wiki/Laplace\%27s_equation}{Laplace's equation} \fi in parabolic coordinates;	 for arbitrary $n$, the solutions of this equation are known as \iftex \textit{parabolic cylinder functions.} \fi \ifblog \href{https://dlmf.nist.gov/12.2#i}{parabolic cylinder functions.}\fi

Another \textit{remarkable} set of functions are the \iftex \textbf{Hermite functions} \fi \ifblog \href{https://en.wikipedia.org/wiki/Hermite_polynomials#Hermite_functions}{Hermite functions} \fi $h_n(x)$, they share a similar relation to \eqref{eq:hermite-conv} with the Chebyshev-Hermite functions $he_n(x)$ given by,

\begin{equation} \label{eq:hermite-fun-conv}
h_n(x) = 2^{\frac{n}{2}} he_n(\sqrt{2}x) = e^{\frac{-x^2}{2}}H_n(x).
\end{equation}

Just as the Chebyshev-Hermite functions, the Hermite functions  $\{h_n(x)\}_{n=0}^{\infty}$ are orthogonal in $L^2(\mathbb{R})$ and they satisfy the following differential equation(also called Hermite equation),

\begin{equation}\label{eq:her-phy-eq}
h''_{n}(x)-x^2h_n(x) = -\left(2n+1\right)h_n(x),
\end{equation}

as it can be directly checked  or derived from \eqref{eq:weber}.

The Hermite functions $h_n(x)$ play a fundamental  role in Quantum Mechanics (see \cite{vlado},\cite{hassani}) and in \iftex Fourier analysis \fi \ifblog \href{https://en.wikipedia.org/wiki/Fourier_analysis}{Fourier analysis} \fi as we will see on Section \ref{sesu:her-phys-four-trans}.

\subsection{Integral representation} \label{sesu:her-int-rep}

We present yet another way of obtaining the Chebyshev-Hermite polynomials. This time, the method depends on the remarkable result that $e^{\frac{-x^2}{2}}$ is its own \iftex Fourier transform. \fi \ifblog \href{https://en.wikipedia.org/wiki/Fourier_transform}{Fourier transform.} \fi

In other words, computing the inverse Fourier transform of $e^{\frac{-k^2}{2}}$ is $e^{\frac{-x^2}{2}}$, this is (using the Fourier Transform defined in \cite{hassani}),

\begin{equation}\label{eq:four-transf-sgd}
e^{\frac{-x^2}{2}} = \mathcal{F}^{-1}\left[ e^{\frac{-k^2}{2}} \right] = \frac{1}{\sqrt{2\pi}}\int_{-\infty}^{\infty}e^{\frac{-k^2}{2}} e^{ixk}dk. 
\end{equation}

Since this result is fundamental for the rest of the section, we present a derivation  in the following example.

\begin{example}\label{exa:gaussian-fourier-tr}
We  compute the following integral,

$$I(x) = \frac{1}{\sqrt{2\pi}}\int_{-\infty}^{\infty}e^{-\frac{k^2}{2}} e^{ixk}dk. $$

Clearly the previous integral is equal to the following integral,

\begin{equation}\label{eq:four-transf-sgd-2}
I(x) = \frac{2}{\sqrt{2\pi}}\int_{0}^{\infty}e^{-\frac{k^2}{2}} \cos(xk)dk. 
\end{equation}

Taking the derivative of $I$ with respect to $x$  and integrating by parts we get,

\begin{eqnarray*}
\frac{dI(x)}{dx} & = & \frac{-2k}{\sqrt{2\pi}}\int_{0}^{\infty}e^{-\frac{k^2}{2}} \sin(xk)dk,\\
&=& -x \left[ \frac{2}{\sqrt{2\pi}}\int_{0}^{\infty}e^{-\frac{k^2}{2}} \cos(xk)dk \right].
\end{eqnarray*}

Then the integral, $I(x)$ satisfies the ordinary differential equation,

$$\frac{dI(x)}{dx}  = -xI(x).$$

The solution of this ODE is $I(x) = C e^{\frac{-x^2}{2}}$, to find $C$, we evaluate \eqref{eq:four-transf-sgd-2} in $x=0$, 

$$C = I(0) = \frac{2}{\sqrt{2\pi}}\int_{0}^{\infty}e^{-\frac{k^2}{2}} dk = 1. $$

So we found that, 

$$I(x) = e^{\frac{-x^2}{2}}.$$

\end{example}
\begin{exercise}\label{ex:gauss-four-pro}
Compute the integral \eqref{eq:four-transf-sgd-2} using (a) contour integration, and (b) by expanding $\cos(xk)$ in powers of $k$ and integrating term by term.
\end{exercise}

Now if we differentiate the integral \eqref{eq:four-transf-sgd-2} $n$ times respect to $x$ we get,
$$\frac{d^n }{d x^n}\left( e^{\frac{-x^2}{2}} \right) = \frac{(i)^{n}}{\sqrt{2\pi}}\int_{-\infty}^{\infty} k^{n} e^{\frac{-k^2}{2}} e^{ixk}dk. $$
Now inserting the previous result in the Rodrigues formula  \eqref{eq:her-prob-rodri} we find the \textit{integral representation} of the Chebyshev-Hermite polynomials,
\begin{equation}\label{eq:her-prob-int-rep}
He_n(x) = \frac{ e^{\frac{x^2}{2}} (-i)^{n}}{\sqrt{2\pi}}\int_{-\infty}^{\infty} k^{n} e^{\frac{-k^2}{2}} e^{ixk}dk.
\end{equation}

From the previous result we can easily find the following Fourier transforms,
\begin{equation}
\mathcal{F}^{-1}\left[ k^n e^{\frac{-k^2}{2}} \right] = i^{n} He_n(x)  e^{\frac{-x^2}{2}}, \qquad \mathcal{F}\left[ x^n e^{\frac{-x^2}{2}} \right] = (-i)^{n} He_n(k)  e^{\frac{-k^2}{2}}.
\end{equation}
With all the previous tools in hand, now we are ready to understand the relation of the raw moments ${\bf E}[Y^n]$ from a random variable $Y\stackrel{d}{=}\mathcal{N}(\hat{x},1)$ with normal distribution and the Chebyshev-Hermite polynomials.

\begin{example}\label{exa:gauss-moments-explicit}
In spirit of \cite{willink} we present a remarkable way to compute the moments of a random variable $\hat{Y}\stackrel{d}{=}\mathcal{N}(\mu,\sigma)$ with a general Gaussian distribution. For this we rewrite the integral representation \eqref{eq:her-prob-int-rep} as follows,
$$He_n(x) (-i)^{n}= \frac{ (-1)^{n}}{\sqrt{2\pi}}\int_{-\infty}^{\infty} k^{n} e^{\frac{-(k-ix)^2}{2}} dk= \frac{ 1}{\sqrt{2\pi}}\int_{-\infty}^{\infty} z^{n} e^{\frac{-(z+ix)^2}{2}}   dz, $$
where we used the change of variable $z=-k$ in the last step. Now we let $x=i\hat{x}$ in the last equation and we recognize  the raw moments ${\bf E}[Y^n]$ from a random variable $Y\stackrel{d}{=}\mathcal{N}(\hat{x},1)$ with normal distribution, 
\begin{equation}\label{eq:her-prob-moments}
He_n(i\hat{x}) (-i)^{n} = \frac{ 1}{\sqrt{2\pi}}\int_{-\infty}^{\infty} z^{n} e^{\frac{-(z-\hat{x})^2}{2}}   dz = {\bf E}[Y^n].
\end{equation}
So we solved part of the mystery of the coefficients of the raw moments of  ${\bf E}[Y^n]$, they are they are all positive as can be seen from \eqref{eq:her-prob-expl}, in fact they are the \textit{absolute value of the coefficients} of the Chebyshev-Hermite polynomials. A simpler way to see this is by doing the change $x\rightarrow i\hat{x}$ and $t\rightarrow -i\hat{t}$ in \eqref{eq:herm-prob-gen} obtaining precisely the moment generating function $m_Y(\hat{t})= e^{\hat{x}\hat{t}+\frac{\hat{t}^2}{2}}$ of a random variable $Y\stackrel{d}{=}\mathcal{N}(\hat{x},1)$.

Now scaling properly \eqref{eq:her-prob-moments}, we can find the the raw moments,  ${\bf E}[\hat{Y}^n]$ using $He_n(x)$ and their explicit expression, the details are left as an exercise for the reader, 
\begin{equation}\label{eq:moment-gauss-expl}
{\bf E}[\hat{Y}^n] =  (-i\sigma)^{n} He_n(i\mu/\sigma), \qquad {\bf E}[\hat{Y}^n] = \sigma^{n} n! \sum_{j=0}^{\left\lfloor{\frac{n}{2}}\right\rfloor}\frac{ (\mu/\sigma)^{n-2j} }{2^j (n-2j)!j!}. 
\end{equation} 
\end{example}

After a long road of working with \textit{Chebyshev-Hermite polynomials} and \textit{Chebyshev-Hermite functions}, we could not resist adding a section entirely to the \textit{remarkable} Fourier transform of \textbf{Hermite functions} \eqref{eq:hermite-fun-conv}. This deserves to be presented to encourage the reader to learn more about the incredible field of Fourier analysis.

\subsection{Fourier Transform of Hermite functions*}\label{sesu:her-phys-four-trans}

\iftex \small \fi

* \emph{This section can be omitted without loss of continuity. It is meant for the ambitious reader or as an interesting general reference}.

\iftex \normalsize \fi

The \iftex Fourier transform \fi \ifblog \href{https://en.wikipedia.org/wiki/Fourier_transform}{Fourier transform} \fi is an essential tool in many areas of math and science, we present here the outstanding result that the Hermite functions $h_n(x)$ \eqref{eq:hermite-fun-conv}, are the \iftex \textit{eigenfunctions} of the Fourier Transform \fi \ifblog\href{https://en.wikipedia.org/wiki/Hermite_polynomials#Hermite_functions_as_eigenfunctions_of_the_Fourier_transform}{eigenfunctions of the Fourier Transform}\fi  integral operator.

To informally see this, we  take the Hermite equation and its Fourier transform, lets write equation \eqref{eq:her-phy-eq} using an arbitrary function $f(x)$,
$$
f''(x)-x^2f(x) = -\left(2n+1\right)f(x).
$$

Using \iftex elementary results \fi \ifblog \href{https://en.wikipedia.org/wiki/Fourier_transform#Tables_of_important_Fourier_transforms}{elementary results} \fi from Fourier transform theory \cite{folland}, we see that the Fourier transform of this equation is,
$$\hat{f}''(k)-k^2\hat{f}(k) = -\left(2n+1\right)\hat{f}(k),
$$

where we denoted $\mathcal{F}[f](k)=\hat{f}(k)$, and took $\xi \rightarrow k$ from the table in \cite{folland}.

This means that both the function $f(x)$ and its Fourier Transform $\hat{f}(k)$ satisfy the same differential equation, \eqref{eq:her-phy-eq}. One possible option can be that $f(x)$ and its Fourier Transform $\hat{f}(k)$ are both Hermite functions $h_n$ and proportional\footnote{In fact more technically, this means that $\mathcal{F}$ preserves each eigenspace of the differential operator $\mathcal{D} = (d^2/dx^2)-x^2$, moreover the Fourier transform operator conspires with the space to restric the possibilites and make the only solution that $f(x)$ and its Fourier Transform $\hat{f}(k)$ are both Hermite functions and proportional, see \cite{reeder} for more detalis.\emo}. However we will see in a moment that this is not the complete answer. 

Now we proceed to find the \iftex Fourier transform of the Hermite functions \fi \ifblog\href{https://web2.ph.utexas.edu/~gleeson/ElectricityMagnetismAppendixB.pdf}{Fourier transform of the Hermite functions}\fi  and to check they are effectively the eigenfunctions of the operator. For a more complete and remarkable treatment of the eigenproblem of the Fourier Transform  see \cite{reeder}. Coming back to our approach first note that by the Fourier inversion theorem, applying twice the Fourier operator yields $\mathcal{F}^2[f](x)=f(-x)$, then $\mathcal{F}^4=I$. 

Next to gain some insight, we consider momentarily the eigenvalue problem for the Fourier linear integral operator, ${\mathcal{F}}[f](k) = \lambda f(k)$. Then if we apply $\mathcal{F}$ three more times to ${\mathcal{F}}[f](k) = \lambda f(k)$ we get,
$$f(k) = I[f](k) = {\mathcal{F}}^4[f](k) = \lambda^4 f(k).$$

This forces $\lambda^4=1$, so the eigenvalues are $1,-1,i,-i$. In fact even functions have ${\mathcal{F}}$-eigenvalues $\pm 1$ and odd functions have ${\mathcal{F}}$-eigenvalues $\pm i$. 

Now we use the generating function of the Hermite functions $h_n(x)$ to find their Fourier transform. First note that the function $\hat{G}(x,t)=e^{2xt-t^2}$ is the generating function for the Hermite polynomials $H_n(x)$ \eqref{eq:hermite-conv}, this can be easily seen by reproducing the procedure in Section  \ref{sesu:her-exp-expre}  with $\hat{G}(x,t)$ and show that one arrives to \eqref{eq:her-phy-expl}. Hence the generating function for the Hermite functions $h_n(x)$ is,

\begin{equation}\label{eq:he-phy-fun-gen}
e^{\left( \frac{-x^2}{2}+2xt-t^2\right)} = \sum_{n=0}^\infty h_n(x)\frac{t^n}{n!} .
\end{equation}

Then we can rewrite the right hand side of the previous equation,
$$e^{\left( \frac{-x^2}{2}+2xt-t^2\right)} = e^{\left( \frac{-x^2+4xt-4t^2}{2}+\frac{2t^2}{2}\right)} = e^{-\frac{(x-2t)^2}{2}}\cdot e^{t^2}.$$

Applying the Fourier transform $\mathcal{F}$  to \eqref{eq:he-phy-fun-gen}, the right hand side yields,

$$\mathcal{F}\left[e^{\left( -\frac{x^2}{2}+2xt-t^2\right)}\right] = e^{t^2} \mathcal{F}\left[e^{-\frac{(x-2t)^2}{2}}\right] =  e^{t^2} e^{-\frac{k^2}{2}} e^{-2ikt} = e^{\left( -\frac{k^2}{2}+2k(-it)-(-it)^2\right)}, $$

where we used the result for a Fourier transform of a Gaussian function \eqref{eq:four-transf-sgd} and the \iftex space shifting \fi \ifblog \href{https://en.wikipedia.org/wiki/Fourier_transform#Tables_of_important_Fourier_transforms}{space shifting} \fi property for the Fourier transform \cite{folland}.

We recognize this Fourier transform as the generating function \eqref{eq:he-phy-fun-gen} with $t\rightarrow -it$ and $x\rightarrow k$, then the Fourier transform of the right and left hand side yields,
$$
 \sum_{n=0}^\infty h_n(k)(-i)^n\frac{t^n}{n!}=\mathcal{F}\left[e^{\left( -\frac{x^2}{2}+2xt-t^2\right)}\right] = \mathcal{F}\left[\sum_{n=0}^\infty h_n(x)\frac{t^n}{n!}\right] = \sum_{n=0}^\infty \mathcal{F}\left[h_n(x)\right]\frac{t^n}{n!}.
$$
Now equating for the coefficients of the series we find,

\begin{equation}
\mathcal{F}\left[h_n\right](k) = (-i)^n  h_n(k).
\end{equation}

Then we see more clearly that the Hermite functions $h_n(k)$ solve the eigenvalue problem  ${\mathcal{F}}[f](k) = \lambda f(k)$ with eigenvalues $1,-1,i,-i$. 

To comprehend why this is the only solution of the eigenvalue problem, see the derivation from \cite{reeder}, in addition, another useful source is  \cite{wiener}. It turns out this remarkable result has some applications, like proving the Heisenberg's inequality \iftex (uncertainty principle)\fi \ifblog \href{https://en.wikipedia.org/wiki/Uncertainty_principle}{(uncertainty principle)} \fi\cite{spe-func}, also one can  prove \iftex Plancherel formula \fi \ifblog \href{https://en.wikipedia.org/wiki/Plancherel_theorem}{Plancherel formula} \fi with Hermite polynomials, see \cite{Hermite-encyclopedia} and references therein.

\subsection{Hermite polynomials in higher dimensions}
\label{sesu:hermite-d-dimensions}
One can generalize the single variable Chebyshev-Hermite polynomials $He_n(x)$ in various ways. An option is to construct a two-parameter family of single variable orthogonal polynomials $He_n^{             \gamma}(x)$ as is done \cite{two-par-hermite}, this is useful for generalizations of the Fourier transform or the quantum harmonic oscillator algebra \cite{rosenblum}. Also, other interesting generalizations are \cite{batahan},\cite{cesarano},\cite{roman}. Now we focus on the generalization of the polynomials to higher dimensions or multiple variables.

If we take the Rodrigues formula \eqref{eq:her-prob-rodri}, we recognize that it can be written as follows\footnote{One can also start from  \eqref{eq:rodri-weight} with different weight functions to construct different families of classical orthogonal polynomials, this is the approach taken by \cite{hassani}.\emo}

\begin{equation}\label{eq:rodri-weight}
He_n(x)=\frac{(-1)^n}{w(x)}\frac{d^n}{dx^n}[w(x)],
\end{equation}

where $w(x)=e^{-\frac{x^2}{2}}$ is the weight function of the Chebyshev-Hermite polynomials.

Now using the usual standard tensor notation, very much like the manuscript \cite{yuan}, let $\mathbf{x}=(x_1,\cdots,x_d)\in \mathbb{R}^d$, we can extend both the definition of the weight function $w(\mathbf{x})$ and the Chebyshev-Hermite polynomials $\mathbf{He}
^{(n)}(\mathbf{x})$ as follows,

\begin{equation}\label{eq:rodri-hermite-weight}
w(\mathbf{x}) = e^{\frac{-\Vert\mathbf{x}\Vert^2}{2}},\qquad \mathbf{He}
^{(n)}(\mathbf{x}) = \frac{(-1)^n}{w(\mathbf{x})}\boldsymbol{\nabla}^{(n)}[w(\mathbf{x})],
\end{equation}

where both  $\mathbf{He}
^{(n)}(\mathbf{x})$ and $\boldsymbol{\nabla^{(n)}}$ are tensors of rank $n$, which can also be written in index notation as $\nabla^{(n)}_{\alpha_1\alpha_2\cdots \alpha_{n-1}\alpha_n}$ where  $\{\alpha_i\}_{i=0}^{n}$ are a list of indexes symbolizing any of the $d$ space coordinates.
\begin{exercise}\label{ex:her-high-prop}
Check the following special cases for the first $d$-dimensional Chebyshev-Hermite polynomials,

\begin{eqnarray*}
He^{(0)}(\mathbf{x}) &=& 1,\\
He^{(1)}_{\alpha_1}(\mathbf{x}) &=& x_{\alpha_1},\\
He^{(2)}_{\alpha_1\alpha_2}(\mathbf{x}) &=& x_{\alpha_1}x_{\alpha_2}-\delta_{\alpha_1\alpha_2},\\
He^{(3)}_{\alpha_1\alpha_2\alpha_3}(\mathbf{x}) &=& x_{\alpha_1}x_{\alpha_2}x_{\alpha_3}-x_{\alpha_1}\delta_{\alpha_2\alpha_3}-x_{\alpha_2}\delta_{\alpha_1\alpha_3}-x_{\alpha_3}\delta_{\alpha_1\alpha_2}.
\end{eqnarray*}
\end{exercise}
\begin{exercise}\label{ex:her-high-prop-gen-rodrig}
Obtain a generalization of the second equation of exercise \ref{ex:her-prop-rodrigues},

\begin{equation}
\mathbf{He}
^{(n)}(\mathbf{x}) = (\mathbf{x}-\boldsymbol{\nabla})^{(n)}\cdot 1, \quad or \quad He^{(n)}_{\alpha_1\alpha_2\cdots \alpha_{n-1}\alpha_n}(\mathbf{x}) =  (x_{\alpha}-{\nabla}_{\alpha})^{(n)}\cdot 1.\medskip
\end{equation}
\end{exercise}

Now we present the generalization of two fundamental properties, Lemma \ref{le:her-recurrence} of the Chebyshev-Hermite polynomials, which can be seen as a special case of the following,

\begin{lemma}\label{le:her-hihg-recurrence} 
The sequence of the d-dimensional Chebyshev-Hermite polynomials, $\left\{\mathbf{He}
^{(n)}(\mathbf{x})\right\}_{n=0}^{\infty}$,  satisfy the two following recurrence relations,
\begin{equation}\label{eq:he-high-re-1}
{\nabla}_{\alpha} He^{(n)}_{\alpha_1\alpha_2\cdots \alpha_{n-1}\alpha_n}(\mathbf{x}) = \sum_{k=1}^n \delta_{\alpha\alpha_{k}}He^{(n-1)}_{\alpha_1\alpha_2\cdots \alpha_{k-1}\alpha_{k+1}\cdots\alpha_n}(\mathbf{x}),
\end{equation}
\begin{equation}\label{eq:he-high-re-2}
{x}_{\alpha} He^{(n)}_{\alpha_1\alpha_2\cdots \alpha_{n-1}\alpha_n}(\mathbf{x}) = He^{(n+1)}_{\alpha\alpha_1\alpha_2\cdots \alpha_{n-1}\alpha_n} + \sum_{k=1}^n \delta_{\alpha\alpha_{k}}He^{(n-1)}_{\alpha_1\alpha_2\cdots \alpha_{k-1}\alpha_{k+1}\cdots\alpha_n}(\mathbf{x}).
\end{equation}
\end{lemma}
\begin{proof}
See \cite{grad} for details.
\end{proof}

For a wider generalization of multivariate Hermite polynomials see \cite{willink}, in this article several results for multivariate normal distributions in relation to the multivariate Hermite polynomials are presented. This whole section is sort of meant to give a brief warm up to understand the former article. 

Next, we present the generalization of the orthogonality relation, Lemma \ref{le:her-prob-orth} of the Chebyshev-Hermite polynomials, which can be seen as a special case of the following,

\begin{lemma} \label{le:her-high-prob-orth}
The d-dimensional Chebyshev-Hermite polynomials are orthogonal with respect to the weight function $w(\mathbf{x}) =e^{\frac{-\Vert\mathbf{x}\Vert^2}{2}}$ in the space $\mathbb{R}^d$. In other words,
\begin{equation} \label{eq:her-high-prob-orth} 
\int_{\mathbb{R}^d} e^{\frac{-\Vert\mathbf{x}\Vert^2}{2}}He^{(n)}_{\boldsymbol{\alpha}}(\mathbf{x})He^{(n')}_{\boldsymbol{\beta}}(\mathbf{x})d\mathbf{x}=({2\pi})^{d/2}\prod_{i=1}^{d} n_i! \boldsymbol{\delta}_{\boldsymbol{\alpha \beta}}^{(n+n')}\delta_{nn'},
\end{equation}
where $\boldsymbol{\alpha}$ is an abbreviation of the subscripts $\alpha_1\alpha_2\cdots \alpha_{n-1}\alpha_n$; $\boldsymbol{\delta}_{\boldsymbol{\alpha \beta}}^{(n+n')}$ is a generalized Kronecker delta, which is unity if the indexes $\boldsymbol{\alpha}$ are a permutation of $\boldsymbol{\beta}$, and zero otherwise. Finally $n_i$ is the number of occurrences of $x_i$ in $\boldsymbol{\alpha}$.
\end{lemma}

\begin{proof}
A detailed proof of this lemma can be found in \cite{grad}, here we give an outline of the reasons of the appearance of each term in the equation following a generalization of the proof of Lemma \ref{le:her-prob-orth}. First the term $({2\pi})^{d/2}$ appears as a result of the normalization of the $d$-dimensional  weight function $w(\mathbf{x})$.

 Both the term $\prod_{i=1}^{d} n_i!$ and $\boldsymbol{\delta}_{\boldsymbol{\alpha \beta}}^{(n+n')}$ appear as a consequence of the repetitive process of partial integration as in the proof of Lemma \ref{le:her-prob-orth}. If the degree of the polynomial for each of the coordinates of $He^{(n)}_{\boldsymbol{\alpha}}(\mathbf{x})$ do not match the degree of the polynomials of $He^{(n)}_{\boldsymbol{\alpha}}(\mathbf{x})$, in other words the indexes $\boldsymbol{\alpha}$ are not a permutation of $\boldsymbol{\beta}$, the iterative integration by parts will yield a zero factor. 
 
The term $\prod_{i=1}^{d} n_i!$ appears applying the argument of the proof  of Lemma \ref{le:her-prob-orth} for each of the leading coefficients from all coordinate variable polynomials $p(\alpha_i)$ appearing in $He^{(n)}_{\boldsymbol{\alpha}}(\mathbf{x})$. Clearly the last term $\delta_{nn'}$ is a consequence of the fact that if $n\neq n'$, then at least one polynomial degree of $He^{(n)}_{\boldsymbol{\alpha}}(\mathbf{x})$ corresponding to a coordinate should be different than the former polynomial in $He^{(n')}_{\boldsymbol{\beta}}(\mathbf{x})$, therefore in this case  the partial integration also generates a zero factor.
\end{proof}

The $d$-dimensional Chebyshev-Hermite polynomials share more properties with the $1$-dimensional version. For example they form a complete basis in $L^2_{w(\mathbf{x})}(\mathbb{R}^d)$, they can also be used to numerically compute integrals in  $\mathbb{R}^d$ via \textit{cubature}(quadrature in higher dimensions) as in Section \ref{sesu:quadrature-1d}. This fact is one of the foundations for the construction of the equilibrium distribution of the \iftex Lattice Boltzmann method \fi \ifblog \href{https://en.wikipedia.org/wiki/Lattice_Boltzmann_methods}{Lattice Boltzmann method} \fi\cite{lattice} for solving numerically PDE(partial differential equations).

\section{Applications to the connection problem of polynomials, probability and  graph theory}\label{se:her-app}

In this section we present some applications of the Chebyshev-Hermite polynomials to the classical connection problem of the theory of polynomials, the representation of densities or functions of random variables using the polynomials and the remarkable connection of graph theory with the Chebyshev-Hermite polynomials.

\subsection{Connection problem of Hermite polynomials and Gaussian moments}
\label{sesu:her-con-pro}

Considering again the raw moments ${\bf E}[Y^n]$ from a random variable $Y\stackrel{d}{=}\mathcal{N}(x,1)$ with normal distribution, we want to explore further on the relation of the moment generating function  $m_Y(t)= e^{xt+\frac{t^2}{2}}$ and the generating function \eqref{def:herm-prob-gen} of the Chebyshev-Hermite polynomials. After digging a little into the problem, we realize we are in the domain of the general connection problem of the theory of polynomials \cite{spe-func},\cite{roman}:

\begin{definition}\iftex{\bf : \cite{con-problem}  Connection problem of $\{S_n(x)\}_{n=0}^{\infty}$ and  $\{P_n(x)\}_{n=0}^{\infty}$}\fi\ifblog <em><span style=\"font-weight: 700;\">: \cite{con-problem}  Connection problem of $\{S_n(x)\}_{n=0}^{\infty}$ and  $\{P_n(x)\}_{n=0}^{\infty}$</span></em>\fi

Given two polynomial sets $\{S_n(x)\}_{n=0}^{\infty}$ and $\{P_n(x)\}_{n=0}^{\infty}$ of $\deg(S_n) = \deg(P_n) = n$. The so-called
\textit{connection problem} between them asks to find the coefficients $C_m(n)$ in the expression:
\begin{equation}\label{eq:conn-equation}
S_n(x) = \sum_{m=0}^n C_m(n) P_m(x). 
\end{equation}
\end{definition}

Without realizing it we have solved this problem for both $\{x^n\}_{n=0}^{\infty}$ , $\{He_n(x)\}_{n=0}^{\infty}$ and $\{(2x)^n\}_{n=0}^{\infty}$ , $\{H_n(x)\}_{n=0}^{\infty}$ in equations \eqref{eq:her-prob-expl}, \eqref{eq:her-prob-inverse} and \eqref{eq:her-phy-expl}, \eqref{eq:her-phys-inverse} respectively. Now if we consider the two polynomials sets as $\{{\bf E}[Y^{n}]({x})\}_{n=0}^{\infty} $ and $\{ He_n({x})\}_{n=0}^{\infty}$, solving the connection problem between the two will be a way to explode the similarities between the moment generating function $m_Y(t)$ and the Chebyshev-Hermite generating function \eqref{eq:herm-prob-gen}. This is containted in the following theorem,

\begin{theorem}\label{th:her-conn-pro}
Let $\{{\bf E}[Y^{n}]({x})\}_{n=0}^{\infty} $ be the raw moments of the random variable \iftex \break \fi $Y\stackrel{d}{=}\mathcal{N}(x,1)$, where $x$ is the variable of the polynomial set, and $\{ He_n({x})\}_{n=0}^{\infty}$ be the Chebyshev-Hermite polynomial set. Then the connection problem between the two sets as defined in \eqref{eq:conn-equation} is solved by the following,
\begin{equation} \label{eq:gauss-her-conn}
He_n(x) = n! \sum_{j=0}^{\left\lfloor{\frac{n}{2}}\right\rfloor}\frac{(-1)^{j} {\bf E}[Y^{n-2j}]({x}) }{(n-2j)!j!},
\end{equation}
\begin{equation} \label{eq:gauss-her-conn-inverse}
{\bf E}[Y^{n}]({x}) = n! \sum_{j=0}^{\left\lfloor{\frac{n}{2}}\right\rfloor} \frac{He_{n-2j}(x)}{(n-2j)!j!}.
\end{equation}
\end{theorem}
\begin{proof}
The main difficulty of this problem is that $\{{\bf E}[Y^{n}]({x})\}_{n=0}^{\infty} $ has no orthogonality relation, hence the technique of Excercise \ref{ex:her-prop-rodrigues} is not applicable to find $He_n(x)$ as an expansion of the raw moment poynomials, for this reason we use a different approach that mainly explotes the relation of the generating functions $m_Y(t)$ and \eqref{eq:herm-prob-gen}.

We prove first \eqref{eq:gauss-her-conn}, for this let's consider the Chebyshev-Hermite polynomials constructed by their generating function \eqref{eq:herm-prob-gen},
$$He_n(x)=\left. \frac{d^n}{dt^n}\left[e^{xt-\frac{t^2}{2}}\right]\right\vert_{t=0}=\left. \frac{d^n}{dt^n}\left[e^{-t^2}m_Y(t)\right]\right\vert_{t=0},$$
where we recognize $m_Y(t)=e^{xt+\frac{t^2}{2}}$ as the moment generating function of $Y\stackrel{d}{=}\mathcal{N}(x,1)$. Now recalling the result known as  \iftex Leibniz Formula \fi \ifblog \href{https://en.wikipedia.org/wiki/General_Leibniz_rule}{Leibniz Formula} \fi \cite{nist-sp-fun},
$$\frac{d^n}{dt^n}\left( fg\right)=\sum_{k=0}^n {{n}\choose{k}} \frac{d^k}{dt^k}\left( f\right)\cdot \frac{d^{n-k}}{dt^{n-k}}\left( g\right).$$
We apply it to the previous equation,
$$He_n(x)=\left.\sum_{k=0}^{n} {{n}\choose{k}} \frac{d^{k}}{dt^{k}}\left( e^{-t^2}\right) \frac{d^{n-k}}{dt^{n-k}}\left( m_Y(t)\right) \right\vert_{t=0}.$$
Now using the Rodrigues formula of $H_k(t)=(-1)^ke^{t^2}\frac{d^{k}}{dt^{k}}\left( e^{-t^2}\right)$, that can be derived along the same lines of Section \ref{sesu:Rodrigues} using the generating function $\hat{G}(t,z)=e^{2tz-z^2}$, and also using the fundamental property $\left. \frac{d^{k}}{dt^{k}} m_Y(t)\right\vert_{t=0}={\bf E}[Y^{k}](x)$,
$$He_n(x)=\sum_{k=0}^{n}\left. {{n}\choose{k}}  (-1)^{k}e^{-t^2}H_k(t)\right\vert_{t=0} {\bf E}[Y^{n-k}](x)=\sum_{k=0}^{n} {{n}\choose{k}}  (-1)^{k}\left.H_k(t)\right\vert_{t=0} {\bf E}[Y^{n-k}](x).$$
It can be seen from \eqref{eq:her-phy-expl} that $H_{2m}(0)=\frac{(-1)^m(2m)!}{m!}$ and $H_{2m+1}(0)=0$, hence making $k=2j$ in the previous equation,
$$He_n(x) =  \sum_{j=0}^{\left\lfloor{\frac{n}{2}}\right\rfloor } {{n}\choose{2j}}  (-1)^{2j} \frac{(-1)^{j}(2j)!}{j!} {\bf E}[Y^{n-2j}](x).$$
Now manipulating the coefficients in the sum,
\begin{eqnarray*}
He_n(x) &=&  \sum_{j=0}^{\left\lfloor{\frac{n}{2}}\right\rfloor} \frac{n!}{(n-2j)!(2j)!}(-1)^{j} \frac{(2j)!}{j!}  {\bf E}[Y^{n-2j}](x),\\
&=& n! \sum_{j=0}^{\left\lfloor{\frac{n}{2}}\right\rfloor}  \frac{(-1)^{j}{\bf E}[Y^{n-2j}](x)}{j!(n-2j)!}.  \\
\end{eqnarray*}
which is precisely the connection formula \eqref{eq:gauss-her-conn}.

For \eqref{eq:gauss-her-conn-inverse} we could expand ${\bf E}[Y^{n}](x)$ as a linear combination of Chebyshev-Hermite polynomials and compute the coefficients with the orthogonality relation \eqref{eq:her-high-prob-orth}, however there is a more insightful and quicker way to obtain it, for this recall the explicit expression of the Hermite polynomials \eqref{eq:her-phy-expl} and the formula we just proved \eqref{eq:gauss-her-conn},

$$H_n(x) = n! \sum_{j=0}^{\left\lfloor{\frac{n}{2}}\right\rfloor}\frac{(-1)^{j} (2x)^{n-2j} }{(n-2j)!j!},\qquad He_n(x) = n! \sum_{j=0}^{\left\lfloor{\frac{n}{2}}\right\rfloor}  \frac{(-1)^{j}{\bf E}[Y^{n-2j}](x)}{j!(n-2j)!}.$$
Note that the coefficients of the connection problem between $\{H_n(x)\}_{n=0}^{\infty}$ and $\{(2x)^n\}_{n=0}^{\infty}$ are the same that of the connection problem between $\{ He_n({x})\}_{n=0}^{\infty}$ and $\{{\bf E}[Y^{n}]({x})\}_{n=0}^{\infty} $. 

So we guess that,
$${\bf E}[Y^{n}]({x}) \overset{?}{=} n! \sum_{j=0}^{\left\lfloor{\frac{n}{2}}\right\rfloor} \frac{He_{n-2j}(x)}{(n-2j)!j!},$$
as with the coefficients of the inverse explicit expression \eqref{eq:her-phys-inverse} for the connection problem of $\{{\bf E}[Y^{n}]({x})\}_{n=0}^{\infty} $ and $\{ He_n({x})\}_{n=0}^{\infty}$.

Indeed this turns out to be the case, to see this we introduce the following notation: Let $p_n(x)=\sum_{i=0}^n a_iH_i(x)$, this polynomial can also be written as \iftex \linebreak\fi $p_n(x)=\sum_{i=0}^n b_i(2x)^i$, where in the first equation the polynomial is written with the ordered base $\{H_i(x)\}_{i=0}^{n}$ and the second with $\{(2x)^i\}_{i=0}^n$. We write the coefficients of $p_n$ with respect	 to each base as the column vectors of size $n+1$,
$$\begin{bmatrix}p_n(x)\end{bmatrix}_H=\begin{bmatrix}a_i\end{bmatrix}_H^T, \qquad \begin{bmatrix}p_n(x)\end{bmatrix}_{2x}=\begin{bmatrix}b_i\end{bmatrix}_{2x}^T, $$
where we emphasize on the fact that the index $i$ stars from zero. Now we want to construct a \iftex change of basis \fi \ifblog \href{https://en.wikipedia.org/wiki/Change_of_basis}{change of basis} \fi  matrix $\boldsymbol{P}_{2x\leftarrow H}$ using equation \eqref{eq:her-phy-expl}, for this fix a $n$ and note that,
\[\begin{bmatrix}H_k(x)\end{bmatrix}_H=
\begin{bmatrix}
0 \\
0 \\
\vdots\\
1\\
\vdots\\
0\\0
\end{bmatrix}_H\begin{matrix}
i=0\phantom{--}\\
i=1\phantom{--}\\
{\vdots}\phantom{--}\\
i=k\phantom{--}\\
{\vdots}\phantom{--}\\
i=n-1\\
i=n\phantom{--}
\end{matrix},\qquad \begin{bmatrix}(2x)^k\end{bmatrix}_{2x}=
\begin{bmatrix}
0 \\
0 \\
\vdots\\
1\\
\vdots\\
0\\0
\end{bmatrix}_{2x}\begin{matrix}
i=0\phantom{--}\\
i=1\phantom{--}\\
{\vdots}\phantom{--}\\
i=k\phantom{--}\\
{\vdots}\phantom{--}\\
i=n-1\\
i=n\phantom{--}
\end{matrix}.
 \]
 With this notation we can write \eqref{eq:her-phy-expl} as,
\begin{equation}\label{eq:coef-vec-c}
\begin{bmatrix}H_k(x)\end{bmatrix}_H=\sum_{i=0}^{n}c_i^{n,k}\begin{bmatrix}(2x)^i\end{bmatrix}_{2x}, \quad \text{where }c_{i}^{n,k}=\begin{cases}\vspace{-10pt}\frac{k!(-1)^{\frac{k-i}{2}}}{(i)!\left(\frac{k-i}{2}\right)!}& \qquad \text{if } i\leq k\text{ and either }  \\ 
& i\text{ and }k \text{ even } \text{or } i\text{ and }k \text{ odd,} \\
0 & \qquad \text{otherwise.} 
\end{cases}
\end{equation}
Now if we set $\boldsymbol{c}^{n,k}=\begin{bmatrix}c_i^{n,k}\end{bmatrix}^T$, we can build the change of base matrix as,
\begin{equation}\label{eq:matrix-hto2x-def}
\boldsymbol{P}_{2x\leftarrow H}=\Bigg[\boldsymbol{c}^{n,1},\boldsymbol{c}^{n,2},\ldots,\boldsymbol{c}^{n,n}\Bigg],
\end{equation}
where its evident from \eqref{eq:coef-vec-c} that the matrix is upper triangular and invertible.

Therefore using this matrix, for any $p_n(x)\in \mathbb{P}_n[x]$ we have,
\begin{equation}\label{eq:matrix-hto2x}
\begin{bmatrix}p_n(x)\end{bmatrix}_{2x}=\boldsymbol{P}_{2x\leftarrow H}\begin{bmatrix}p_n(x)\end{bmatrix}_H.
\end{equation}
Clearly $\boldsymbol{P}_{H\leftarrow 2x}$ can be constructed similarly as \eqref{eq:coef-vec-c} with \eqref{eq:her-phys-inverse}, obtaining,
\begin{equation}\label{eq:coef-vec-d}
\boldsymbol{P}_{H\leftarrow 2x}=\Bigg[\boldsymbol{d}^{n,1},\boldsymbol{d}^{n,2},\ldots,\boldsymbol{d}^{n,n}\Bigg],\quad \text{where }d_{i}^{n,k}=\begin{cases}\vspace{-10pt}\frac{k!}{(i)!\left(\frac{k-i}{2}\right)!}& \qquad \text{if } i\leq k\text{ and either }  \\ 
& i\text{ and }k \text{ even } \text{or } i\text{ and }k \text{ odd,} \\
0 & \qquad \text{otherwise.}  
\end{cases}
\end{equation}
Furthermore using this matrix, for any $p_n(x)\in \mathbb{P}_n[x]$ we have,
\begin{equation}\label{eq:matrix-2xtoh}
\begin{bmatrix}p_n(x)\end{bmatrix}_{H}=\boldsymbol{P}_{H\leftarrow 2x}\begin{bmatrix}p_n(x)\end{bmatrix}_{2x}.
\end{equation}
Observing \eqref{eq:matrix-2xtoh} and \eqref{eq:matrix-hto2x} we conclude,
\begin{equation}\label{eq:matrix-change-inverse}
\boldsymbol{P}_{H\leftarrow 2x} = \Big(\boldsymbol{P}_{2x\leftarrow H}\Big)^{-1},
\end{equation}
as is well known from the theory of vector spaces\cite{goode}.

Now focusing specifically on the connection problem of  $\{{\bf E}[Y^{n}]({x})\}_{n=0}^{\infty} $ and $\{ He_n({x})\}_{n=0}^{\infty}$,  let $p_n(x)=\sum_{i=0}^n \alpha_iHe_i(x)$, this polynomial can also be written as $p_n(x)=\sum_{i=0}^n \beta_i{\bf E}[Y^{i}]({x})$, where in the first equation the polynomial is written with the ordered base $\{He_i(x)\}_{i=0}^{n}$ and the second with $\{{\bf E}[Y^{n}]({x})\}_{i=0}^n$. We write the coefficients of $p_n$ with respect	 to each base as the column vectors of size $n+1$,
$$\begin{bmatrix}p_n(x)\end{bmatrix}_{He}=\begin{bmatrix}\alpha_i\end{bmatrix}_{He}^T, \qquad \begin{bmatrix}p_n(x)\end{bmatrix}_{{\bf E}[Y]}=\begin{bmatrix}\beta_i\end{bmatrix}_{{\bf E}[Y]}^T. $$
Similarly as we did for $\boldsymbol{P}_{2x\leftarrow H}$, using \eqref{eq:gauss-her-conn} fixing a $n$ we find the change of basis matrix $\boldsymbol{P}_{{\bf E}[Y]\leftarrow He}$ is,
\begin{equation}\label{eq:matrix-hetogauss-def}
\boldsymbol{P}_{{\bf E}[Y]\leftarrow He}=\Bigg[\boldsymbol{c}^{n,1},\boldsymbol{c}^{n,2},\ldots,\boldsymbol{c}^{n,n}\Bigg]=\boldsymbol{P}_{2x\leftarrow H},
\end{equation}
where  the column vectors $\boldsymbol{c}^{n,k}$ are the same as in \eqref{eq:matrix-hto2x-def}, since the coefficients of both connection problems, \eqref{eq:gauss-her-conn} and \eqref{eq:her-phy-expl} are the same.

In addition using this matrix, for any $p_n(x)\in \mathbb{P}_n[x]$ we have,
\begin{equation}\label{eq:matrix-hetogauss}
\begin{bmatrix}p_n(x)\end{bmatrix}_{{\bf E}[Y]}=\boldsymbol{P}_{{\bf E}[Y]\leftarrow He}\begin{bmatrix}p_n(x)\end{bmatrix}_{He}.
\end{equation}
Using \eqref{eq:matrix-change-inverse}, \eqref{eq:matrix-hetogauss-def} and recalling that the inverse of matrix is unique,
$$\boldsymbol{P}_{H\leftarrow 2x} = \Big(\boldsymbol{P}_{2x\leftarrow H}\Big)^{-1}=\Big(\boldsymbol{P}_{{\bf E}[Y]\leftarrow He}\Big)^{-1}=\boldsymbol{P}_{He\leftarrow {\bf E}[Y]}.$$
Hence we conclude,
\begin{equation}
\boldsymbol{P}_{He\leftarrow {\bf E}[Y]}=\Bigg[\boldsymbol{d}^{n,1},\boldsymbol{d}^{n,2},\ldots,\boldsymbol{d}^{n,n}\Bigg],
\end{equation}
where  the column vectors $\boldsymbol{d}^{n,k}$ are the same as in \eqref{eq:coef-vec-d}. Now from the construction of each column in the change of basis matrix $\boldsymbol{P}_{He\leftarrow {\bf E}[Y]}$ we have,
$$\begin{bmatrix}{\bf E}[Y^k](x)\end{bmatrix}_{{\bf E}[Y]}=\sum_{i=0}^{n}d_i^{n,k}\begin{bmatrix}He_i(x)\end{bmatrix}_{He}  \Longleftrightarrow  {\bf E}[Y^k](x) = \sum_{i=0}^n d_i^{n,k}He_i(x).$$
Writting the coefficients from \eqref{eq:coef-vec-d} compactly we find, 
$${\bf E}[Y^{k}]({x}) = k! \sum_{j=0}^{\left\lfloor{\frac{k}{2}}\right\rfloor} \frac{He_{k-2j}(x)}{(k-2j)!j!},$$
which is precisely \eqref{eq:gauss-her-conn-inverse} with $n\rightarrow k$.
\end{proof}

Now a little exercise for the reader to apply the method that helps prove the second part of Theorem \ref{th:her-conn-pro}, where we use coefficient comparison with an already stablished connection problem.

\begin{exercise} \label{ex:gauss-hermite-conv-2}
Show that the moments ${\bf E}[He_n(Y)]$ of  a random variable $Y\stackrel{d}{=}\mathcal{N}(x,1)$ with normal distribution are given by,
\begin{equation}\label{eq:moments-hermite}
{\bf E}[He_n(Y)](x)= n! \sum_{j=0}^{\left\lfloor{\frac{n}{2}}\right\rfloor}\frac{(-1)^{j} {\bf E}\left[Y^{n-2j}\right](x) }{2^j (n-2j)!j!} = x^n.
\end{equation}
Suggestion: Write the expected value using the explicit expression \eqref{eq:her-prob-expl} and use the linearity of the expectated value. Next use the fact that the coefficients of the connection problem between $\{He_n(x)\}_{n=0}^\infty$ and $\{x^n\}_{n=0}^\infty$ from \eqref{eq:her-prob-expl} and \eqref{eq:her-prob-inverse} are the same that of the connection problem between $\{{\bf E}\left[Y^{n}\right](x)\}_{n=0}^\infty$ and $\{x^n\}_{n=0}^\infty$  from \eqref{eq:moment-gauss-expl} with $\sigma=1$ and $\mu=x$ , hence obtaining the same result as in \eqref{eq:herm-convolution}.

\end{exercise}

\subsection{Representation of  densities and functions via Hermite polynomials}

In this section with the aid of the Chebyshev-Hermite polynomials, we describe two ways in which one can represent an arbitrary density or function of random variables as a combination of simpler densities or simpler functions of random variables.

\subsubsection{Fourier-Hermite expansion}\label{sesusu:Fourier-Hermite}

Consider an arbitrary probability density $f(x)\in L^2(\mathbb{R})$, we could expand this function in a series of Chebyshev-Hermite functions \eqref{eq:chebyshev-hermite-functions}, however we folow \cite{singer} and take a related approach\footnote{ The expansion  \eqref{eq:moment-expansion} can be written in terms of the Chebyshev-Hermite functions \eqref{eq:chebyshev-hermite-functions} noting that $w(x)^{\frac{1}{2}}He_n(x)=he_n(x)$. Therefore the function $g(x)=w(x)^{\frac{-1}{2}}f(x)$  can be used to write $f(x)$ since its expansion is $g(x)=\sum_{n=0}^{\infty} a_n he_n(x)$, thus the completeness of the expansion of the Chebyshev-Hermite functions justifies the expansion \eqref{eq:moment-expansion}. \emo}. We expand $f(x)$ as,
\begin{equation}\label{eq:moment-expansion}
f(x)=w(x)\sum_{n=0}^{\infty} a_n He_n(x)=e^{\frac{-x^2}{2}}\sum_{n=0}^{\infty} a_n He_n(x).
\end{equation}

Using the orthogonality of the Chebyshev-Hermite polynomials \eqref{eq:her-prob-orth}, we obtain the following expression for the coefficients,
\begin{equation}\label{eq:fo-her-coef}
a_n=\frac{1}{\sqrt{2\pi}n!}\int_\infty^\infty He_n(x)f(x)dx=\frac{1}{\sqrt{2\pi}n!}{\bf E}\left[He_n(X)\right],
\end{equation}
where $X$ is a random variable with density $f(x)$. Next we consider a simple example to get familiar with the expansion,
\iftex\vspace{-10pt}\fi
\begin{example}\label{exa:four-her-gauss-std}

Let's suppose that  $X\stackrel{d}{=}\mathcal{N}(\mu,1)$. Correspondingly its density is given by $f(x)=\frac{1}{\sqrt{2\pi}}e^{\frac{-(x-\mu)^2}{2}}$, computing the fourier-hermite coefficients \eqref{eq:fo-her-coef},
\iftex\vspace{-10pt}\fi
$$a_n=\frac{1}{\sqrt{2\pi}n!}{\bf E}\left[He_n(X)\right] = \frac{1}{\sqrt{2\pi}n!} \mu^n,$$
where we used \eqref{eq:moments-hermite}. Now using the expansion \eqref{eq:moment-expansion},
\iftex\vspace{-10pt}\fi
$$f(x)=e^{\frac{-x^2}{2}}\sum_{n=0}^{\infty} \frac{1}{\sqrt{2\pi}n!} \mu^n He_n(x)=\frac{e^{\frac{-x^2}{2}}}{\sqrt{2\pi}}G(x,\mu)=\frac{e^{\frac{-x^2}{2}}}{\sqrt{2\pi}}e^{\mu x-\frac{\mu^2}{2}}=\frac{1}{\sqrt{2\pi}}e^{\frac{-(x-\mu)^2}{2}},$$
\iftex
\vspace{-15pt}\fi

where we recognized the series expansion of the generating function \eqref{def:herm-prob-gen}, thus recovering the original density from the Fourier-Hermite expansion.
\end{example}
\iftex
\vspace{-5pt}
\fi
Back in the general case we can obtain a series in terms of the standardized moments(\iftex cumulants\fi \ifblog \href{https://en.wikipedia.org/wiki/Cumulant}{cumulants}\fi) of $X$. For this consider the standardized variable $Z=(X-\mu)/\sigma$, with ${\bf E}\left[X\right]=\mu$, $\sigma^2={\bf E}\left[X^2\right]-\mu^2$, we have ${\bf E}[Z]=0$,${\bf E}[Z^2]=1$, ${\bf E}[Z^n]:=\nu_n$. Now using the table in Figure \ref{fig:hermite}, \eqref{eq:moment-expansion} and \eqref{eq:fo-her-coef} we obtain the expansion,
\iftex
\vspace{-5pt}
\fi
$$p(z)=\frac{e^{\frac{-z^2}{2}}}{\sqrt{2\pi}}(1+(1/6)\nu_3He_3(z)+(1/24)(\nu_4-3)He_4(z)+\ldots).$$
Recalling that we can write the density of a function of a random variable in terms of the original random variable density (Theorem 2.4 \cite{liliana}), we conclude  $p(x)=\frac{1}{\sigma}p(z)$ with $z=(x-\mu)/\sigma$, this allows us to write the expansion \eqref{eq:moment-expansion} as,
\iftex
\vspace{-5pt}
\fi
\begin{equation}
p(x)=\frac{e^{\frac{-z^2}{2}}}{\sqrt{2\pi}\sigma}(1+(1/6)\nu_3He_3(z)+(1/24)(\nu_4-3)He_4(z)+\ldots),
\end{equation}
commonly known as the \iftex \textit{Gram-Charlier} \fi \ifblog \href{https://www.encyclopediaofmath.org/index.php/Gram-Charlier_series}{Gram-Charlier} \fi expansion in the statistics literature \cite{johnson}. Even though this expansion is frequently used, it \textit{diverges} for many situations of interest. For this reason other expansions like the \textit{Gauss-Hermite} series or the \iftex\textit{Edgeworth}\hspace{-2pt}\fi\ifblog \href{https://en.wikipedia.org/wiki/Edgeworth_series}{Edgeworth} \fi asymptotic series are more convenient to describe nearly Gaussian distributions \cite{blinnikov}.

Now following the treatment of \cite{luo}, we proceed to an interesting analogy of the previous result to functions of random variables. Consider a random variable $Z$ defined as a function $Z=f(Y)$, where $Y\stackrel{d}{=}\mathcal{N}(0,1)$, we expand this random variable similarly as,
\begin{equation}\label{eq:moment-expansion-RV-1}
f(Y)=\sum_{n=0}^{\infty} b_n He_n(Y),
\end{equation}
where the coefficients are given by,
\begin{equation}\label{eq:moment-expansion-RV-2}
b_n=\frac{1}{n!}{\bf E}\left[He_n(Y)f(Y)\right]=\frac{1}{n!}\int_\infty^\infty He_n(y)f(y)\frac{e^{\frac{-y^2}{2}}}{\sqrt{2\pi}}dy.
\end{equation}
Similarly we can generalize the previous result to $d$-dimensions, where we note \iftex \linebreak\fi $\vec{Y}=(Y_1,\ldots,Y_d)$ as a random vector of unit gaussian random variables. Considering now the random vector $\vec{Z}$ defined as the function $\vec{Z}=f\left(\vec{Y}\right)$, and using the $d$-dimensional Chebyshev-Hermite polynomials, Section \ref{sesu:hermite-d-dimensions}, we expand the function,  
\begin{equation}\label{eq:moment-expansion-RV-1-d}
f\left(\vec{Y}\right)=\sum_{n=0}^{\infty} {\bf b}^{(n)}\cdot {\bf He}^{(n)}\left(\vec{Y}\right),
\end{equation}
\begin{equation}\label{eq:moment-expansion-RV-2-d}
{\bf b}^{(n)}=\frac{1}{n!}{\bf E}\left[{\bf He}^{(n)}\left(\vec{Y}\right)f\left(\vec{Y}\right)\right]=\frac{1}{n!}\int_{\mathbb{R}^d} {\bf He}^{(n)}\left(\vec{y}\right)f(\vec{y})\frac{e^{\frac{-\Vert \vec{y}\Vert^2}{2}}}{(2\pi)^{d/2}}d\vec{y},
\end{equation}
where the expansion coefficients ${\bf b}^{(n)}$ are tensors of rank $n$, and the dot product ${\bf b}^{(n)}\cdot {\bf He}^{(n)}$ is defined as the contraction $b^{(n)}_{\alpha_1\cdots  \alpha_n}He^{(n)}_{\alpha_1\cdots  \alpha_n}$.

The expansion \eqref{eq:moment-expansion-RV-2-d} and the coefficients \eqref{eq:moment-expansion-RV-1-d}  constitute a fundamental part of the so-called technique of \iftex \textit{ Wiener Chaos Expansion}(WCE),\fi \ifblog \href{https://en.wikipedia.org/wiki/Polynomial_chaos}{Wiener Chaos Expansion}(WCE), \fi this technique is heavily used to analyze \iftex Stochastic Partial Differential Equations(SPDEs). \fi \ifblog \href{https://en.wikipedia.org/wiki/Stochastic_partial_differential_equation}{Stochastic Partial Differential Equations(SPDEs).} \fi  These equations are relevant when we model complex phenomena with "randomness", e.g. diffusion through heterogenous random media or fluid motion with random forcing.

An example of this type of equations is the stochastic difussion equation,
\begin{equation}
\frac{\partial u}{\partial t}=\Delta u +  \dot{W}(t),
\end{equation} 
where $W(t)$ is a \iftex Brownian motion, \fi \ifblog \href{https://en.wikipedia.org/wiki/Brownian_motion#Mathematics}{Brownian motion,} \fi  and the derivative of the Brownian motion $\dot{W}(t)$ is the so-called \iftex \textit{continuous-time white noise}. \fi \ifblog \href{https://en.wikipedia.org/wiki/White_noise#Continuous-time_white_noise}{continuous-time white noise.}\fi

In this problem we see that $u$ not only depends on $(x,t)$, but also the brownian motion path, since we consider a time dependent white noise, $u$ is a functional of $\{W(s),0\leq s \leq t\}$, that is all the possible Brownian motions up to time $t$.

It turns out that since for a fixed $t>0$, a Brownian motion path $\{W(s),0\leq s \leq t\}$ \iftex  \linebreak \fi can be decomposed as a linear combination of infinite unit gaussian random variables,  $W(s)=W(s;Y_1,Y_2,\ldots)$, then we can write the solution as,
\begin{equation}
u(x,t)=U(x,t;Y_1,Y_2,\ldots).
\end{equation}
In a similar way as the multidimensional Fourier-Hermite expansion \eqref{eq:moment-expansion-RV-2-d}, we can expand the solution $u$ using a family of polynomials derived from the Chebyshev-Hermite polynomials, called the \iftex \textit{Wick} polynomials.\fi \ifblog \href{https://en.wikipedia.org/wiki/Wick_product}{Wick  polynomials. } \fi The previous construction roughly conforms the mentioned  \iftex \textit{ Wiener Chaos Expansion}(WCE),\fi \ifblog \href{https://en.wikipedia.org/wiki/Polynomial_chaos}{Wiener Chaos Expansion}(WCE), \fi for more details see \cite{luo}.

The power of the WCE is that one represents the randomness by a set of random bases with deterministic coefficients, thus obtaining a spectral expression in the probabilistic space. This expansion can also be used to solve \textit{numerically} SPDEs, where the complete expansion is truncated to compute an approximation. For an example of this procedure using the \iftex Finite Element Method(FEM) \fi \ifblog \href{https://en.wikipedia.org/wiki/Finite_element_method#Technical_discussion}{Finite Element Method(FEM)} \fi see \cite{galvis}.
\subsubsection{A mixture of Gaussian distributions}

Instead of considering a density as a series of Chebyshev-Hermite functions as we did in the previous section, we follow \cite{jaynes} and consider the subsequent problem,
\begin{question*}
Is a non-Gaussian distribution explainable as a mixture of Gaussian ones ? 
\end{question*}
This problem can be attacked in a few ways and it is solvable to some extent in each of them depending on how the statement is interpreted. 

If we consider the problem of estimating a density function using a set of initial data points sampled from a process with such density, we can estimate the density as a sum of Gaussian distributions. This constitutes the usual method of \iftex Kernel density estimation \fi \ifblog\href{https://en.wikipedia.org/wiki/Kernel_density_estimation}{Kernel density estimation}\fi (KDE) \cite{silverman}, this is a well known numerical technique that has been extensively studied and with its aid one can solve to a certain degree the aforementioned question, for more details see \cite{wand},\cite{botev},\cite{keith}.

On the other hand in this section we approach the problem in a different way, assume that we have a well defined density $g(y)$, and we want to see if this density can be written as a mixture of Gaussians in the following sense,
\begin{equation}\label{eq:density-conv}
g(y) = (\phi_\sigma * f)(y) = \int_{\mathbb{R}}f(x)\frac{e^{\frac{-(x-y)^2}{2\sigma^2}}}{\sqrt{2\pi}\sigma}dx,
\end{equation}
where $\phi_\sigma(y)$ is a Gaussian density with mean zero and variance $\sigma^2$.

In other words we have interpreted the question as, 
\begin{question**}
Can we find a mixing function $f(x)$ such that its convolution with $\phi_\sigma(y)$ \eqref{eq:density-conv}, yields the density $g(y)$ ?
\end{question**}
We can see that by taking $f(x)=\sum_{i=1}^N\delta(y-y_i)$, \eqref{eq:density-conv} becomes a linear combination of $N$ gaussians, each centered in $y_i$. This superposition is the usual form that the  method of KDE assumes(supposing that $\{y_i\}_{i=1}^N$ comes from a dataset)\cite{silverman}. However we require more information of $g(y)$ to recover it completly. To obtain a closed form answer of the solution to the reformulated question we can use the Chebyshev-Hermite polynomials, first assume $g(y)$ has the following power series representation,
$$g(y)=\sum_{n=0}^\infty a_ny^n.$$
Using the convolution property of the Chebyshev-Hermite polynomials \eqref{eq:herm-convolution}, we obtain the formal solution,
\begin{equation}
f(x)=\sum_{n=0}^\infty a_n\sigma^nHe_n\left(\frac{x}{\sigma}\right).
\end{equation}
We can check that this is indeed the case replacing the previous expresion in the Gaussian convolution \eqref{eq:density-conv},
\begin{eqnarray*}
(\phi_\sigma * f)(y) &=& \int_{\mathbb{R}}\sum_{n=0}^\infty a_n\sigma^nHe_n\left(\frac{x}{\sigma}\right)\frac{e^{\frac{-(x-y)^2}{2\sigma^2}}}{\sqrt{2\pi}\sigma}dx,\\
&=& \sum_{n=0}^\infty a_n \sigma^n \int_{\mathbb{R}} He_n\left(\frac{x}{\sigma}\right)\frac{e^{\frac{-\left(\frac{x}{\sigma}-\frac{y}{\sigma}\right)^2}{2}}}{\sqrt{2\pi}}d\left(\frac{x}{\sigma}\right),\\
&=& \sum_{n=0}^\infty a_n \sigma^n \left(\frac{y}{\sigma}\right)^n,\\
&=& g(y),
\end{eqnarray*}
where we use the convolution property from line two to three of the previous manipulation. We can write this solution in a little more insightful way if we use the explicit expression \eqref{eq:her-prob-expl}, recall then that,
$$He_n\left(\frac{x}{\sigma}\right) = n! \sum_{j=0}^{\left\lfloor{\frac{n}{2}}\right\rfloor}\frac{(-1)^{j}}{2^j (n-2j)!j!}\left(\frac{x}{\sigma}\right)^{n-2j}.$$
Now using the identity,
$$\frac{n!}{(n-2j)!}\left(\frac{x}{\sigma}\right)^{n-2j}=\sigma^{2j-n} \frac{d^{2j}}{dx^{2j}}\left(x^n\right),$$
we obtain the expansion,
\begin{equation}
f(x)=\sum_{j=0}^\infty \frac{(-1)^{j}\sigma^{2j}}{2^j j!}\frac{d^{2j}}{dx^{2j}}g(x)=g(x)-\frac{\sigma^2}{2}\frac{d^{2}}{dx^{2}}g(x)+\frac{\sigma^4}{8}\frac{d^{4}}{dx^{4}}g(x)-\cdots.
\end{equation}
From this expansion we see that the size of the variance has an important impact on the mixture function, similar to the effect of the \iftex \textit{bandwidth} \fi \ifblog\href{https://en.wikipedia.org/wiki/Kernel_density_estimation#Bandwidth_selection}{bandwidth}\fi parameter on the KDE. Additionaly we observe that the higher the order of the expansion, the more we are capturing the details of $g(y)$, to see this let's introduce the operator $D:=\frac{d}{dx}$, then the previous equation becomes,
\begin{equation}
f(x)=e^{\frac{-\sigma^2 D^2}{2}}g(x)=\sum_{j=0}^\infty \left(\frac{-\sigma^2}{2}\right)^j\frac{1}{j!}\frac{d^{2j}g(x)}{dx^{2j}},
\end{equation}

which is the inverse of a \iftex generalized \fi \ifblog\href{https://en.wikipedia.org/wiki/Hermite_polynomials\#Generalizations}{generalized}\fi \iftex Weierstrass-Gauss transform \fi \ifblog\href{https://rmarcus.info/blog/2016/09/09/weierstrass-transform.html}{Weierstrass-Gauss transform}\fi \cite{roman}. We can check that  \eqref{eq:density-conv} coincides with the  Weierstrass-Gauss transform\cite{zayed} for $\sigma = \sqrt{2}$, hence a parallel with image processing arise since the  Weierstrass-Gauss transform is used as a low-pass  \iftex gaussian filter \fi \ifblog\href{https://en.wikipedia.org/wiki/Gaussian_blur}{gaussian filter}\fi to blur images \cite{gonzalez}. Thereby the higher the size of the pass $\sigma$, the more we are capturing the details of $g(y)$. The appearance of the blurring and of course the Gaussian distributions all over the place, hints to a possible \iftex relation \fi \ifblog \href{http://wwwf.imperial.ac.uk/~pavl/lec_fokker_planck.pdf}{relation}\fi to the \iftex Heat \fi \ifblog \href{https://en.wikipedia.org/wiki/Heat_equation}{Heat}\fi \iftex Diffusion \fi \ifblog\href{https://en.wikipedia.org/wiki/Diffusion_equation}{Diffusion}\fi  equation or more generally the \iftex Focker-Planck-Kolmogorov(FPK) \fi \ifblog \href{https://en.wikipedia.org/wiki/Fokker\%E2\%80\%93Planck_equation}{Focker-Planck-Kolmogorov(FPK)}\fi equation, for a treatment on this matter see \cite{pavliotis},\cite{risken}.

Finally, we see that if $g(y)$ is discontinuous, the manipulation we just did is not valid, for a discussion on this matter and the nature of the general problem with relation to its mathematical and probabilistic character see \cite{jaynes}.

\subsection{Combinatorial properties of a Simple Graph}

This section is merely a desire to share our appreciation of the unexpected connections between different mathematical fields. \textit{Who would have thought that orthogonal polynomials have a relation with Graph Theory?} so when we discovered this connection we had to put it in our manuscript.

This connection allows us to solve the so-called linearization problem of Chebyshev-Hermite polynomials and also allows us to evaluate integrals of products of these polynomials. To expose these ideas we follow the presentation of \cite{spe-func}, but first, we introduce a non-exhaustive background of the necessary tools from Graph Theory.

\subsubsection{Background on Graph theory}

We start from the very beginning and  define a \iftex graph \fi \ifblog\href{https://en.wikipedia.org/wiki/Graph_(discrete_mathematics)#Graph}{graph}\fi with a set-theoretical approach,
\iftex
\vspace{-15pt}
\fi
\begin{definition}\label{def:graph-set-def}
A \textbf{Graph} $G$ is an ordered pair $(V,E)$, where $V$ is the set of vertices and $E$  is the set of edges, formed by pair of vertices. The vertex set is a finite non-empty set. The edge set can be empty, but aside from this case, its elements are two-element subsets of the vertex set.
\end{definition}

We denote the size of $V$, the number of vertices as $\vert v\vert$ and the size of $E$, the number of edges as $\vert e \vert$. Next, we stress a little on the set definition of graphs and  complement it following our "intuition" with the graphical representation. After this short "formal" treatment, we rely the rest of this section mostly on the graphical representation.
\iftex
\newpage
\fi
\begin{example}\label{exa:elem-exa-graph}
 The following ordered pairs are some examples of graphs,
\iftex
\vspace{-5pt}
\fi
\begin{itemize}
\item $G_1=(V_1,E_1)$, with $V_1=\{v_1,v_2,v_3,v_4,v_5,v_6\}$, $E_1=\{\{v_2,v_3\},\{v_3,v_4\},\{v_5,v_6\}\}$.
\item $G_2=(V_2,E_2)$, with $V_2=\{v_1,v_2,v_3,v_4,v_5,v_6,v_7\}$,  $E_2=\{\{v_1,v_2\},\{v_1,v_3\},\{v_2,v_4\},$ $\{v_2,v_5\},\{v_3,v_6\},\{v_3,v_7\}\}$ is a \iftex \textbf{tree}. \fi \ifblog\href{https://en.wikipedia.org/wiki/Tree_(graph_theory)}{tree.}\fi
\item $C_4=(V_3,E_3)$, with $V_3=\{v_1,v_2,v_3,v_4\}$, $E_3=\{\{v_1,v_2\},\{v_2,v_3\},\{v_3,v_4\},\{v_4,v_1\}\}$ is called a \iftex \textbf{cyclic graph} \fi \ifblog\href{https://en.wikipedia.org/wiki/Cycle_graph}{cyclic graph} \fi with four vertices.
\item $K_4=(V_4,E_4)$, with $V_4=\{v_1,v_2,v_3,v_4\}$, $E_4=\{\{v_1,v_2\},\{v_1,v_3\},\{v_1,v_4\},\{v_2,v_3\},$ $\{v_2,v_4\},\{v_3,v_4\}\}$ is called a \iftex \textbf{complete graph} \fi \ifblog\href{https://en.wikipedia.org/wiki/Complete_graph}{complete graph} \fi with four vertices.
\end{itemize}
\end{example}
\iftex
\vspace{-10pt}
\fi
We can represent a graph in a visual way, for vertices we draw thick dots and for edges we draw arcs connecting pairs of distinct vertices. For the previous examples, we have the following graphical representations,
\iftex
\vspace{-10pt}
\fi
\ifblog
\begin{figure}[h!]
\image{width = 824}{https://keithpatarroyo.files.wordpress.com/2018/09/graphs_vshort-v21.png}{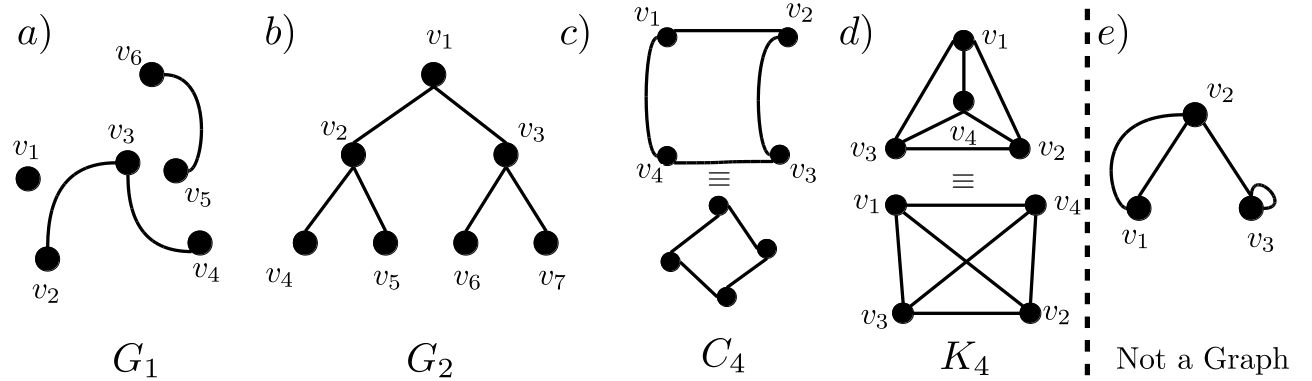}
\captionof{figure}{Examples of valid and invalid graphs. \emo}\label{fig:graphs}
\end{figure}
\fi
\iftex
\begin{figure}[H]
\floatbox[{\capbeside\thisfloatsetup{capbesideposition={left,center},capbesidewidth=1.3cm}}]{figure}[\FBwidth]
{\caption{\color{white}.}\label{fig:graphs}}
{\includegraphics[width=14cm]{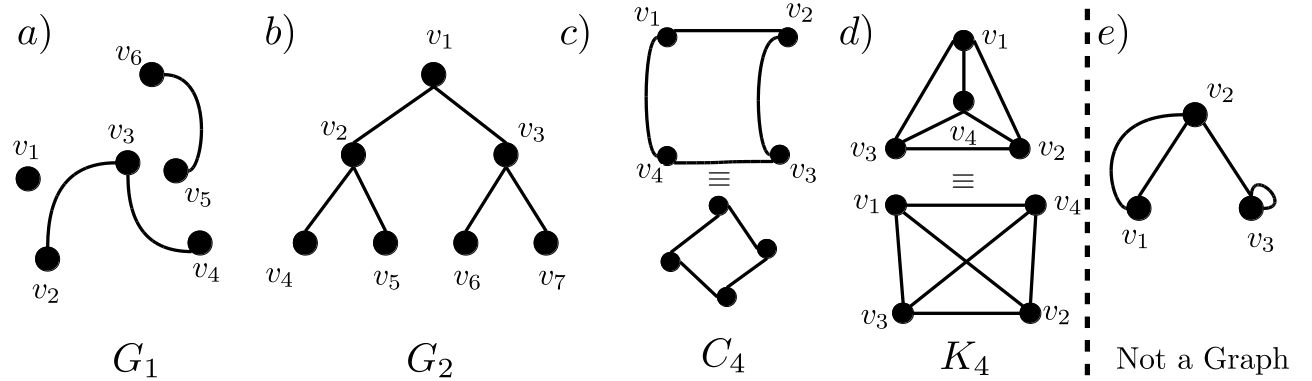}}
\end{figure}
\fi
\iftex
\vspace{-10pt}
\fi
As we can see from the previous graphs, the positions of the vertices and the shapes of the edges are irrelevant. This is, since the only things that matter are the number of vertices and the pairwise relation between them.
\begin{note}\label{note:graphs}
\iftex
\vspace{-5pt}
\fi
We notice the following both visually and symbolically from the set defintion,
\begin{itemize}
\item We do not allow loops in graphs like in  Figure \ref{fig:graphs}.e), since the set $\{v_3\}$ is not allowed in the set of edges.
\item We do not allow multiple edges like in Figure \ref{fig:graphs}.e), since the element $\{v_1,v_2\}$ cannot be allowed to be twice in the set of edges.
\item Since the elements in the set of edges are sets and not ordered pairs, the edges lack a directionality, and hence the graphs that we consider are \iftex \textbf{un-directed}. \fi \ifblog\href{https://en.wikipedia.org/wiki/Directed_graph}{un-directed.} \fi
\item A graphical representation of a graph can have many labelings, so a specific labeling is not of fundamental importance, we stress on the importance of the \underline{\textbf{pairwise relation}} between vertices, which an unlabeled graph still maintains\footnote{Strictly speaking a change in label represent an \iftex isomorphism \fi\ifblog \href{http://mathworld.wolfram.com/Isomorphism.html}{isomorphism} \fi between two graphs\cite{trudeau}. To some extent, the graphical representation becomes the fundamental object, as if it was a person and a new labeling is a new suit, in most practical matters the suit is irrelevant, since it is always the same person, i.e. it maintains the underlying structure(edge, vertex) unchanged.\emo}, like in Figure \ref{fig:graphs}.c).
\end{itemize}
\end{note}
The graphs that we have described up to this point are usually called \iftex \textit{simple graphs,} \fi \ifblog\href{https://en.wikipedia.org/wiki/Graph_(discrete_mathematics)#Simple_graph}{simple graphs,} \fi from now on we will speak of a graph mainly thinking in its graphical representation and we will  not refer to its set definition, unless it is strictly necessary. Now we follow with a couple more definitions,
\begin{definition}
A  \iftex \textbf{path} \fi \ifblog\href{https://en.wikipedia.org/wiki/Path_(graph_theory)}{path} \fi is a finite sequence of distinct edges, each of which has a vertex in common with the next. A path is said to be between two vertices or to connect them if the two vertices are associated with the first or the last edge of the path and no other edge in the sequence.
\end{definition}

We clearly see in Figure \ref{fig:graphs}.a) that in $G_1$ there is a path between $v_2$ and $v_4$, but there is no path between $v_3$ and $v_5$.

\begin{definition}
A graph is called a \iftex \textbf{connected} \fi \ifblog\href{https://en.wikipedia.org/wiki/Graph_(discrete_mathematics)#Connected_graph}{connected} \fi graph if for every two vertices in the graph, there is a path that connects them.
\end{definition}

In Figure \ref{fig:graphs} we see that all graphs are connected excepting $G_1$.

\begin{definition}
A connected graph is called a \iftex \textbf{tree} \fi \ifblog\href{https://en.wikipedia.org/wiki/Tree_(graph_theory)}{tree} \fi, if it contains no closed path, where a path is called closed if two different non-consecutive edges in the path sequence have a vertex in common.
\end{definition}

There is only one tree in  Figure \ref{fig:graphs}, as a curious fact, the trees of the form of $G_2$ are of fundamental importance in \textit{coding theory} and \textit{algorithmics}, we refer the interested reader to \cite{yeung} and \cite{brassard} respectively. Next, we continue with a couple more definitions.

\begin{definition}
If $\vert v \vert$ is a positive integer, the \iftex \textbf{complete graph} \fi \ifblog\href{https://en.wikipedia.org/wiki/Complete_graph}{complete graph} \fi $K_{\vert v \vert}$ on $\vert v \vert$ vertices is a graph where every pair of vertices is connected by an edge.
\end{definition}

In Figure \ref{fig:graphs}.c) we have two graphical representations of $K_4$, in the following figure we give a graphical representation of the first five complete graphs,
\ifblog
\begin{figure}[h!]
\image{width = 624}{https://keithpatarroyo.files.wordpress.com/2018/08/complete-graphs.png}{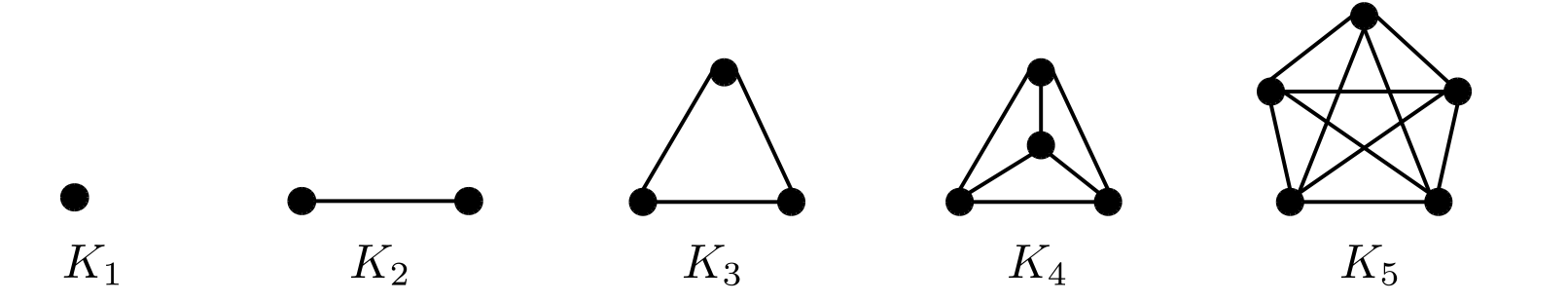}
\captionof{figure}{First five complete graphs $K_{\vert v \vert}$. \emo}\label{fig:comp-graphs}
\end{figure}
\fi
\iftex
\begin{figure}[H]
\floatbox[{\capbeside\thisfloatsetup{capbesideposition={left,center},capbesidewidth=1.3cm}}]{figure}[\FBwidth]
{\caption{\color{white}.}\label{fig:comp-graphs}}
{\includegraphics[width=14cm]{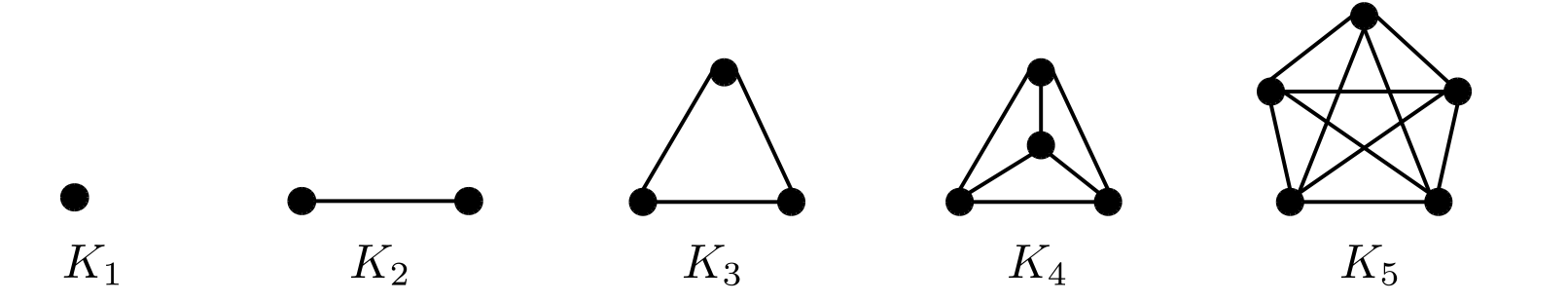}}
\end{figure}
\fi
Complete graphs are a very important part in the relation between graph theory and the Chebyshev-Hermite polynomials, we explore this later in this section. Also complete graphs have other features, in particular $K_5$ is of fundamental importance  in the theory of non-planar graphs\cite{bona}. Now we define what might be the most important notion of this section.
\begin{definition}
A  \iftex \textbf{$j$-match }or \textbf{independent edge set of size $j$} \fi \ifblog $j$-\href{http://mathworld.wolfram.com/IndependentEdgeSet.html}{match} or an \textbf{independent edge set of size $j$} \fi is a set of $j$ disjoint edges, where by disjoint edges we mean that the edges have no vertices in common.
\end{definition}
As an example we see in Figure \ref{fig:graphs} that $\{\{v_1,v_2\},\{v_3,v_4\}\}$ is a $2$-match of $C_4$ and  $\{\{v_1,v_3\},\{v_2,v_4\}\}$ is also a $2$-match in $K_4$, in addition there is no $3$-match in any of the graphs from Figures \ref{fig:graphs} and \ref{fig:comp-graphs}.

Now we are almost ready to introduce the polynomial associated with a graph called \iftex \textit{Matching Polynomial}, \fi \ifblog \href{http://mathworld.wolfram.com/IndependentEdgeSet.html}{Matching Polynomial}, \fi this polynomial accounts for the enumerative properties of the graph. We follow then with the final definitions,
\begin{definition}\label{def:k-matches}
Let $G$ be an arbitrary graph, we denote with $\underline{p(G,j)}$ as the total number of $j$-matches in $G$. We take $p(G,0)=1$ and $p(G,-1)=0$.
\end{definition}
Note that a fixed labeling is neccesary to properly count the total number of $j$-matches in a graph. Also notice that $p(G,1)=\vert e \vert$, in addition $p(G,j)=0$, if $j>\nu(G)$, where $\nu(G)$ is the size of the largest possible $j$-match in a graph $G$, hence  $\nu(G)\leq \lfloor{\frac{\vert v \vert }{2}}\rfloor$.
\begin{definition}
A \iftex \textbf{complete match} \fi \ifblog \href{http://mathworld.wolfram.com/PerfectMatching.html}{complete match} \fi or a \iftex \textbf{perfect match} \fi \ifblog \href{http://mathworld.wolfram.com/PerfectMatching.html}{perfect match} \fi is a $j$-match that uses every vertex of $G$,  the total number of perfect matches is denoted as $p_m(G)$.
\end{definition}

We can see in Figure \ref{fig:graphs} that $C_4$ has $p_m(C_4)=p(C_4,\nu(C_4))=2$, where \iftex \break \fi $\nu(C_4)=\lfloor{\frac{\vert v \vert }{2}}\rfloor=2$. Finally we arrive to the definition of the matching polynomial,

\begin{definition}
The \iftex \textbf{matching polynomial} \fi \ifblog \href{http://mathworld.wolfram.com/IndependentEdgeSet.html}{Matching Polynomial} \fi of  $G$ with $\vert v \vert=m$ is defined by,
\begin{equation}
\mu(G)=\alpha(G)=\alpha(G,x)=\sum_{j=0}^{\nu(G)}(-1)^jp(G,j)x^{m-2j}.
\end{equation}
\end{definition}
With the previous tools at hand we are ready to prove the theorem that connects graph theory with the Chebyshev-Hermite polynomials,
\begin{theorem}
The matching polynomial of a complete graph on $\vert v \vert=m$ vertices is the Chebyshev-Hermite polynomial of order $m$, that is 
\begin{equation}
\alpha(K_m,x)=He_m(x).
\end{equation}
\end{theorem}

\begin{proof}
In the case of a complete graph $\nu(G) =\lfloor{\frac{m }{2}}\rfloor $, to see this choose an arbitrary vertex, then pick an arbitrary edge connecting this vertex. Next, excluding the vertices associated with the previous edge, choose a vertex  and an arbitrary edge connecting this vertex and not connected to the ones previously selected. This procedure can be repeated $\lfloor{\frac{m }{2}}\rfloor $ times, since in a complete graph, for any vertex one can always find another vertex not been previously used, connected to it, therefore $\nu(G)=\lfloor{\frac{m }{2}}\rfloor$.

Next, recall from equation \eqref{eq:her-prob-expl},
$$He_m(x) = \sum_{j=0}^{\left\lfloor{\frac{m}{2}}\right\rfloor}\frac{(-1)^{j}m! x^{m-2j} }{2^j (m-2j)!j!}.$$
Hence it is enough to show that,
$$P(K_m,j)=\frac{m!}{2^j (m-2j)!j!}.$$
Therefore we have to find the number of $j$-matches in a complete graph on $m$ vertices. First counting the $j$-matches, from $m$ vertices we can choose $2j$ vertices in ${m}\choose{2j}$ ways to form a $j$-match using  $2j$ vertices. Next for the matchings utilizing the $2j$ vertices, choose arbitrarily a vertex, we can form an edge with any  $2j-1$ remaining vertices excluding itself, later we choose other vertex(excluding the previous two) and form an edge with any $2j-3$ remaining vertices, continuing with this procedure we find,
\begin{eqnarray*}
P(K_m,j)&=&{{m}\choose{2j}} (2j-1)\cdot(2j-3)\cdot\ldots\cdot(5)\cdot(3)\cdot(1),\\
&=&\frac{m!}{(2j)!(m-2j)!}\cdot \frac{(2j)(2j-1)\cdot\ldots\cdot(3)\cdot(2)\cdot(1)}{(2j)(2j-2)\cdot\ldots\cdot(4)\cdot(2)\cdot(1)},\\
&=&\frac{m!(2j)!}{(2j)!(m-2j)!}\cdot \frac{1}{2^j j!},\\
&=&\frac{m!}{2^j (m-2j)!j!}.
\end{eqnarray*}
Thus the statement is proven.\end{proof}

This completes our short review of graph theory, we suggest the interested reader to see \cite{bona},\cite{trudeau} for a more comprehensive background and \cite{levit} for a more specific review of polynomials on graphs. Now we follow with the problem of integrals of products of Chebyshev-Hermite polynomials, the linearization problem of Chebyshev-Hermite polynomials and how can we use graph theory to solve them.

\subsubsection{Integrals of products of Hermite polynomials}

Suppose we would like to obtain an expression for the multiplication of two Chebyshev-Hermite polynomials of different order, we would like for this resulting polynomial to be a linear combination of Chebyshev-Hermite polynomials, that is 
\begin{equation}\label{eq:lin-problem}
He_n(x)He_m(x)=\sum_{l=0}^{m+n}a(l,m,n)He_l(x).
\end{equation}
Finding the coefficients $a(l,m,n)$ is the so-called \textit{linearization problem} of the Chebyshev-Hermite polynomials. Using the orthogonality of the Chebyshev-Hermite polynomials \eqref{eq:her-prob-orth}, we obtain,
\begin{equation}\label{eq:lin-coeff}
a(l,m,n)=\frac{1}{\sqrt{2\pi}l!}\int_{-\infty}^{\infty} e^{\frac{-x^2}{2}}He_{l}(x)He_m(x)He_{n}(x)dx.
\end{equation}
Therefore we have formulated the linearization problem as the computation of an integral of three products of Chebyshev-Hermite polynomials. To solve this problem we use all the apparatus of graph theory, in fact we employ it to solve the more general case, 
\begin{equation}\label{eq:herm-int-prod}
J(n_1,n_2,\ldots,n_k)=\int_{-\infty}^{\infty} e^{\frac{-x^2}{2}}He_{n_1}(x)He_{n_2}(x)\cdots He_{n_k}(x)dx,
\end{equation}
where $J(n_1,n_2)$ is basically the orthogonality relation  \eqref{eq:her-prob-orth} and $J(n_1,n_2,n_3)$ is essentially \eqref{eq:lin-coeff}. These kind of integrals are not only useful for the considered problem, in fact they are of fundamental importance in the theory of \iftex \textit{molecular vibrations} \fi \ifblog \href{https://en.wikipedia.org/wiki/Molecular_vibration}{molecular vibrations} \fi in Quantum Mechanics, see \cite{arfken} for more on this topic.

The idea of using graph theory to solve this kind of integrals is: Analyze certain type of graphs that have an enumerative property analogous to a recursive property of these integrals. This property  who belongs to an underlying mathematical structure(algebraic structure) allows us to compute the integrals in an indirect way.

Next we introduce the following notation,
$$\vec{n}:=(n_1,\ldots,n_k),\: J^{(i)}_{\vec{n}}:=J(n_1,\ldots ,n_{i-1},n_i-1,n_{i+1},\ldots,n_k),$$
$$J^{(ij)}_{\vec{n}}:=J(n_1,\ldots ,n_{i-1},n_i-1,n_{i+1},\ldots,n_{j-1},n_j-1,n_{j+1},\ldots,n_k).$$
Consequently we state in the following lemma the recursive property of the integrals,
\iftex
\vspace{-5pt}
\fi
\begin{lemma} \label{le:herm-int-recurrence}
The integral of products of Chebyshev-Hermite polynomials satisfy the following recurrence relation,
\begin{equation}\label{eq:herm-int-recurrence}
J_{\vec{n}}=\sum_{i=2}^k n_i J_{\vec{n}}^{(1i)}, \quad \text{with} \quad J_{\vec{0}}=\sqrt{2\pi}.
\end{equation}
\end{lemma}

\begin{proof}
Clearly $J_{\vec{0}}=J(0,0,\ldots,0)$ is the normal integral, since $He_0(x)=1$, then
$$J_{\vec{0}}=\int_{-\infty}^{\infty} e^{\frac{-x^2}{2}}dx=\sqrt{2\pi}.$$
Now for the general case we use a technique similar to that of the proof from Lemma \ref{le:her-prob-orth},\iftex\break\fi first we write the term $He_{n_1}(x)$ using Rodrigues formula \eqref{eq:her-prob-rodri}, then $J_{\vec{n}}$ is,
$$J_{\vec{n}}=\int_{-\infty}^{\infty} (-1)^{n_1} \frac{d^{n_1} }{d x^{n_1}}\left( e^{\frac{-x^2}{2}}\right)He_{n_2}(x)\cdots He_{n_k}(x)dx.$$
Integrating by parts,
\begin{eqnarray*}
J_{\vec{n}}&=&\int_{-\infty}^{\infty} (-1)^{n_1-1} \frac{d^{n_1-1} }{d x^{n_1-1}}\left( e^{\frac{-x^2}{2}}\right) \frac{d }{d x}\left[He_{n_2}(x)\cdots He_{n_k}(x)\right]dx,\\
&=&\int_{-\infty}^{\infty}e^{\frac{-x^2}{2}}e^{\frac{x^2}{2}} (-1)^{n_1-1} \frac{d^{n_1-1} }{d x^{n_1-1}}\left( e^{\frac{-x^2}{2}}\right)\sum_{i=2}^k He'_{n_i}(x)\prod_{\substack{j=2\\j \neq i}}^k He_{n_j}(x)dx,
\end{eqnarray*}
where we used the fact that $e^{\frac{-x^2}{2}}$ and all its derivatives goes to zero for infinite $x$ and the product rule. Now using the recurrence relation \eqref{eq:he-re-1},
\begin{eqnarray*}
J_{\vec{n}}&=&\sum_{i=2}^kn_i\int_{-\infty}^{\infty}e^{\frac{-x^2}{2}} He_{n_1-1}(x)He_{n_i-1}(x)\prod_{\substack{j=2\\j \neq i}}^k He_{n_j}(x)dx,\\
&=&\sum_{i=2}^k n_i\int_{-\infty}^{\infty}e^{\frac{-x^2}{2}}He_{n_1-1}(x)He_{n_2}(x)\cdots He_{n_{i-1}}(x)He_{n_i-1}(x)He_{n_{i+1}}(x)\cdots He_{n_k}(x) dx,\\
&=&\sum_{i=2}^k n_i J_{\vec{n}}^{(1i)},
\end{eqnarray*}
\iftex
\vspace{-20pt}
\fi

where we used the Rodrigues formula again, and the definition of $J_{\vec{n}}^{(ij)}$, therefore the lemma is proven.
\end{proof}

Next we define a very special graph, whose number of complete matches has the same recursive propery of the integrals, i.e. it also satisfies the recurrence relation of the previous lemma,
\begin{definition}
The \iftex \textbf{complete $k$-partite} or \textbf{multipartite graph} \fi \ifblog \href{http://mathworld.wolfram.com/Completek-PartiteGraph.html}{complete} $k$-\href{http://mathworld.wolfram.com/Completek-PartiteGraph.html}{partite} or \href{https://en.wikipedia.org/wiki/Multipartite_graph}{multipartite graph}\fi on $V$, \iftex \linebreak \fi where  $V_1 \cup V_2 \cup \cdots \cup V_k=V$, and $V_1,\ldots,V_k$ are pairwise disjoint vertex sets, is the graph constructed from $V$, in which we put an edge between any pair of vertices that do not belong to the same $V_i$. We denote $\vert V_i \vert=n_i$ and $\vert V \vert=\sum_{i=1}^kn_i$.
\end{definition}
In the following figure we show a couple of examples of $k$-partite graphs,
\iftex
\vspace{-10pt}
\fi
\ifblog
\begin{figure}[h!]
\image{width = 624}{https://keithpatarroyo.files.wordpress.com/2018/09/partite-graphs.png}{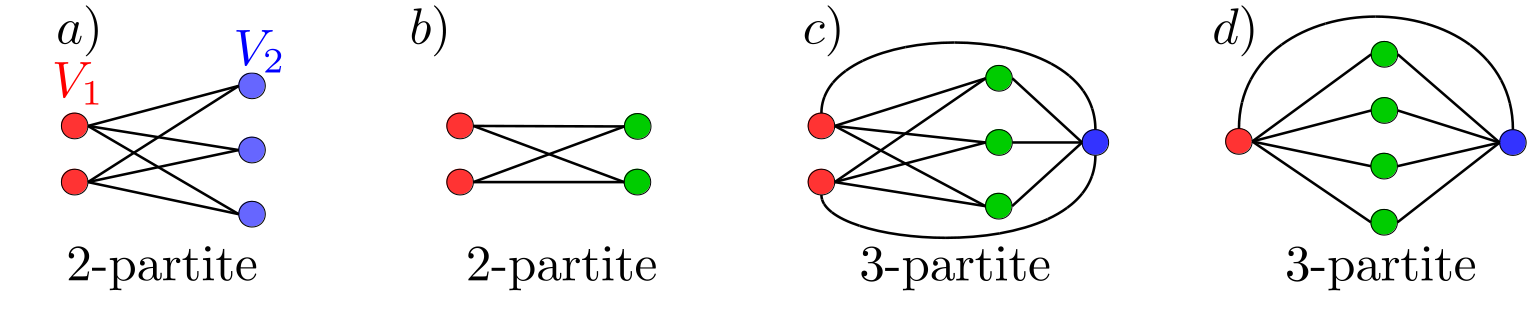}
\captionof{figure}{Examples of $2$-partite and $3$-partite graphs. \emo}\label{fig:k-partite-graphs}
\end{figure}
\fi
\iftex
\begin{figure}[H]
\floatbox[{\capbeside\thisfloatsetup{capbesideposition={left,center},capbesidewidth=1.3cm}}]{figure}[\FBwidth]
{\caption{\color{white}.}\label{fig:k-partite-graphs}}
{\includegraphics[width=14cm]{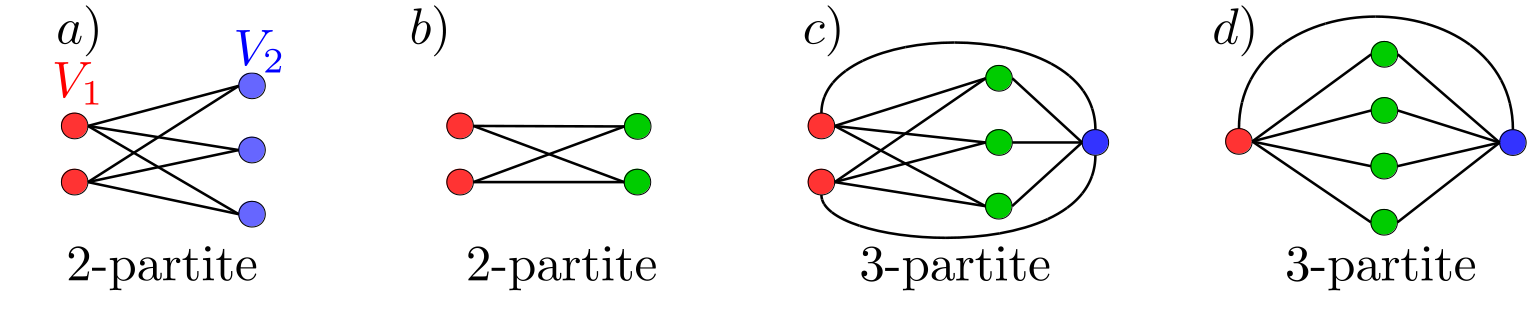}}
\end{figure}
\fi
\iftex
\vspace{-20pt}
\fi
We denote with $P(n_1,n_2,\ldots,n_k)=P_{\vec{n}}$, as the total number of \textit{complete} $k$-matches in a $k$-partite graph, we set $P_{\vec{0}}=1$ according to Definition \ref{def:k-matches}. This number, that encodes certain enumerative features of the graph, is the property that shares the algebraic structure with $J_{\vec{n}}$.

As an example we see in Figure \ref{fig:k-partite-graphs} we have a) $P_{\vec{n}}=0$, b) $P_{\vec{n}}=1$, c) $P_{\vec{n}}=6$, d) $P_{\vec{n}}=0$. In general for an odd $\vert V \vert$, $P_{\vec{n}}=0$, the argument can be easily seen from Figure \ref{fig:k-partite-graphs}.a).

Similarly to $J_{\vec{n}}$ we note,
$$ P^{(i)}_{\vec{n}}:=P(n_1,\ldots ,n_{i-1},n_i-1,n_{i+1},\ldots,n_k),$$
$$P^{(ij)}_{\vec{n}}:=P(n_1,\ldots ,n_{i-1},n_i-1,n_{i+1},\ldots,n_{j-1},n_j-1,n_{j+1},\ldots,n_k).$$
From this notation we can speculate that a similarity will appear with $J_{\vec{n}}$ as we analyse the recursive properties of $P_{\vec{n}}$ when we vary $\vec{n}$. Indeed we show now that $P_{\vec{n}}$ also satisfies the same recurrence relation \eqref{eq:herm-int-recurrence},
\begin{lemma} \label{le:k-match-recurrence}
The total number of complete $k$-matches in a $k$-partite graph satisfies the following recurrence relation,
\begin{equation}\label{eq:k-match-recurrence}
P_{\vec{n}}=\sum_{i=2}^k n_i P_{\vec{n}}^{(1i)}, \quad \text{with} \quad P_{\vec{0}}=1.
\end{equation}
\end{lemma}
\begin{proof}
Choose an arbitrary vertex in $V_1$, then pair it with an abitrary vertex in $V_i$, if we fix the vertex in $V_1$, then we would  have $n_i$ ways to choose the vertex in $V_i$. After this construction, the number of ways to generate a complete match is $P_{\vec{n}}^{(1i)}$, since we have to account for all the posibilities, we add all the ways we could have chosen another vertex in the other $k$ vertex sets, then,
$$P_{\vec{n}}=\sum_{i=2}^k n_i P_{\vec{n}}^{(1i)}.$$
For this proof we fixed the vertex in $V_1$, but clearly we could have also fixed the vertex in $V_i$ and obtained another relation, i.e. we could have counted the matches in this situation as well. Since  $P_{\vec{0}}=1$ by definition, we are done.
\end{proof}

Finally we obtain the last bridge between $J_{\vec{n}}$ and $P_{\vec{n}}$, if we recognize that both functions satisfy the same recurrence relation and apply induction a couple of times, the following result follows,
\begin{theorem} \label{th:k-partite-int-recurrence}
The functions $J_{\vec{n}}$ and $P_{\vec{n}}$ are related by,
\begin{equation}\label{eq:k-partite-int-recurrence}
J_{\vec{n}}=\sqrt{2\pi}P_{\vec{n}}.
\end{equation}
\end{theorem}

An alernative proof of the previous theorem is discussed in \cite{spe-func}. From this result we can go ahead and compute any integral with products of Chebyshev-Hermite polynomials if we know how to compute $P_{\vec{n}}$ first, thereby transfering the most challenging part to graph theory and using the common algebraic framework to terminate it. Next we refocus in a couple special cases of $J_{\vec{n}}$ and compute $P_{\vec{n}}$ for these cases.
\begin{theorem} \label{th:2-partite-k-match}
The number of complete matches for a $2$-partite graph with $\vert V_1 \vert =n$ and $\vert V_2 \vert =m$ is given by,
\begin{equation}\label{eq:2-partite-k-match}
P(m,n)=m!\delta_{mn},
\end{equation}
when $m+n$ is even, and $P(m,n)=0$ otherwise.
\end{theorem}
\begin{proof}
If $m\neq n$, a complete match cannot be achieved since there would always be a vertex that is not used given any $k$-match. If $m=n$, clearly is the same problem of insterting $m$ different balls in $m$ different one-ball boxes, therefore there are $m!$ ways,
$$P(m,n)=m!\delta_{mn},$$
this proves the lemma.
\end{proof}

With equations \eqref{eq:herm-int-prod}, \eqref{eq:k-partite-int-recurrence} and \eqref{eq:2-partite-k-match} we obtain the orthogonality relation of the Chebyshev-Hermite polynomials,
$$J(m,n)=\int_{-\infty}^{\infty} e^{\frac{-x^2}{2}}He_m(x)He_{n}(x)dx=\sqrt{2\pi}P(m,n)=\sqrt{2\pi} m! \delta_{mn}.$$
Reproducing the results of \eqref{eq:her-prob-orth}. With the previous experience on hand, we are ready to resume our solution to the linearization problem of the Chebyshev-Hermite polynomials and compute \eqref{eq:lin-coeff}, this result is a consequence of the following theorem,
\begin{theorem} \label{th:3-partite-k-match}
The number of complete matches for a $3$-partite graph with $\vert V_1 \vert =l$, $\vert V_2 \vert =m$, $\vert V_3 \vert =n$ and $s=\frac{l+m+n}{2}$ is given by,
\begin{equation}\label{eq:3-partite-k-match}
P(l,m,n)=\frac{l!m!n!}{(s-l)!(s-m)!(s-n)!},
\end{equation}
when $l+m+n$ is even and if $l,m,n$ satisfy the so-called triangle property, that is the sum of any two of $l,m,n$ is not greater than the third. $P(l,m,n)=0$ is zero otherwise.
\end{theorem}
\begin{proof}
As noted before if $\vert V \vert =l+m+n$ is odd, then $P(l,m,n)=0$, if $l,m,n$ do not satisfy the triangle property $P(l,m,n)$ is also zero, the general argument can be established analyzing Figure \ref{fig:k-partite-graphs}.d).

Now assume without loss of generality that $m\geq n$, clearly when we have matched all the vertices of $V_1$ with $V_2$ and $V_3$, the same number of vertices in $V_2$ and $V_3$ must remain, that is if $x$ denotes the number of $V_1,V_2$ pairs and $y$ the number of $V_1,V_3$ pairs, we have,
$$m-x=n-y \Longrightarrow  m-n=x-y\qquad \text{and}\quad m\geq n\Rightarrow x\geq y.$$
Hence there are more $V_1,V_2$ pairs than $V_1,V_3$ pairs, in fact there are $m-n$ more pairs. Also we have $x+y=l$, finding $x$ and $y$ explicitly from this and the previous equality,
$$2x=l+m-n \Longrightarrow x=\frac{l+m-n}{2}+n-n=\frac{l+m+n}{2}-n=s-n,$$
$$2y=l+n-m \Longrightarrow y=\frac{l+n-m}{2}+m-m=\frac{l+m+n}{2}-m=s-m.$$
Since we have a complete $k$-match, there are $s$ total pairs, recalling that $x+y=l$, then the total number of $V_2,V_3$ pairs equals $s-l$.

Having developed the previous notation, we can start counting the total number of complete matches. If we consider $V_1$, we have first to choose $x=s-n$ vertices, then we have $l\choose{s-n}$ ways to pair them with vertices in $V_2$. Similarly for vertices of $V_2$ with vertices in $V_3$,$m\choose{s-l}$ and likewise vertices of $V_3$ with vertices in $V_1$,$n\choose{s-m}$. 

Considering the two sets $V_1$ and $V_2$ similarly as with the two sets from the proof of  Theorem \ref{th:2-partite-k-match}, we have $(s-n)!$ ways to do the pairings, furthermore for $V_1$ and $V_3$, we have $(s-m)!$ ways and for $V_2$ and $V_3$, we have $(s-l)!$ ways, all the previous implies,
\begin{eqnarray*}
P(l,m,n)&=&{l\choose{s-n}}{m\choose{s-l}}{n\choose{s-m}}(s-n)!(s-m)!(s-l)!,\\
&=&\frac{l!m!n!(s-n)!(s-m)!(s-l)!}{(s-n)!(l-s+n)!(s-l)!(m-s+l)!(s-m)!(n-s+m)!},\\
&=&\frac{l!m!n!}{(l-s+n+m-m)!(m-s+l+n-n)!(n-s+m+l-l)!},\\
&=&\frac{l!m!n!}{(s-m)!(s-n)!(s-l)!}.
\end{eqnarray*}
Therefore we are done.\end{proof}

Using \eqref{eq:k-partite-int-recurrence} and \eqref{eq:3-partite-k-match} we obtain, 
$$J(l,m,n)=\int_{-\infty}^{\infty} e^{\frac{-x^2}{2}}He_l(x)He_m(x)He_{n}(x)dx=\frac{\sqrt{2\pi}l!m!n!}{(s-m)!(s-n)!(s-l)!},$$
\iftex
\vspace{-5pt}
\fi
if $(l,m,n)$ satisfy the triangle property,  $m+n+l$ is even, and $J(l,m,n)=0$ otherwise. 

Using this equation we can finally solve the linearization problem of the Chebyshev-Hermite polynomials \eqref{eq:lin-problem}, \eqref{eq:lin-coeff},
\iftex
\vspace{-5pt}
\fi
\begin{equation}\label{eq:lin-formula-sol}
a(l,m,n)=\frac{m!n!}{(s-m)!(s-n)!(s-l)!}, \quad He_m(x)He_{n}(x) = \sum_{l=0}^{m+n} a(l,m,n)He_{l}(x),
\end{equation}
\iftex
\vspace{-10pt}
\fi
where the same restricions apply to $a(l,m,n)$ for different values of $l,m,n$.

We can obtain a nicer looking formula by letting $j=s-l$ and noting that,
$$l=m+n-2j,\quad s-m=n-j,\quad s-n=m-j.$$
Next to find constraints on $j$, we apply the constraints of $(m,l,n)$, first since $m+l+n$ is even, we have that $s \in \mathbb{N}$ and as a consequence $j \in \mathbb{N}$ .Also by using the triangle property we find,
$$0\leq s-m\leq n-j \Rightarrow j\leq n,\quad 0\leq s-n\leq m-j \Rightarrow j\leq m,\quad 0\leq s-l=j.$$
Therefore we find that $j\leq \min(m,n)$, note that if $l=0$, then $m+n=2j$; this implies that by taking as last value $j=\min(m,n)$, then $2j$ would only reach $m+n$ if $m=n$. In addition if $l=m+n$ then $j=0$, all of the previous implies that \eqref{eq:lin-formula-sol} can be written as,
$$He_m(x)He_{n}(x) = \sum_{j=0}^{\min(m,n)} \frac{m!n!j!He_{m+n-2j}(x)}{(m-j)!(n-j)!j!j!}.$$
Using the definition of the binomial coefficient we obtain a good looking expression for the linearization problem of the Chebyshev-Hermite polynomials,
\begin{equation}\label{eq:lin-formula-sol-2}
He_m(x)He_{n}(x) = \sum_{j=0}^{\min(m,n)}{m\choose{j}} {n\choose{j}} j!He_{m+n-2j}(x)
\end{equation}
Thus reproducing the known result from \cite{nist-sp-fun}. For two alternative solutions to the linearization problem of the Chebyshev-Hermite polynomials see \cite{spe-func} or \cite{luo}. A third form is proposed in the last exercise,
\begin{exercise} \label{lin-formula-sol-3}
Obtain \eqref{eq:lin-formula-sol} or \eqref{eq:lin-formula-sol-2} using  the product of  two generating functions \eqref{eq:herm-prob-gen} of the Chebyshev-Hermite polynomials, $G(x,t_1)\cdot G(x,t_2)$.

Suggestion: Compare the coefficients that are obtained in the multiplication of the series that represent the generating functions and the expansion of the explicit multiplication, $e^{x(t_1+t_2)-\frac{(t_1+t_2)^2}{2}}\cdot e^{t_1t_2}$.

\end{exercise}

At this point we end the discussion about the Chebyshev-Hermite polynomials, we tried to be as broad as possible to cover different fields where the polynomials take a significant role. Sadly some interesting statistics discussions regarding \iftex Pearson systems \fi\ifblog \href{https://en.wikipedia.org/wiki/Pearson_distribution}{Pearson systems} \fi  or \iftex Edgeworth series \fi\ifblog \href{https://en.wikipedia.org/wiki/Edgeworth_series}{Edgeworth} \fi were not considered, for these see \cite{blinnikov},\cite{johnson}. Also both the relation between the \iftex normal correlation function \fi\ifblog\href{https://en.wikipedia.org/wiki/Correlation_and_dependence\#Correlation_and_independence}{normal correlation function}\fi of the \iftex Bivariate Normal distribution \fi \ifblog\href{https://en.wikipedia.org/wiki/Multivariate_normal_distribution\#Bivariate_case}{Bivariate Normal distribution}\fi and the \iftex Mehler formula \fi \ifblog\href{https://en.wikipedia.org/wiki/Mehler_kernel\#Probability_version}{Mehler formula}\fi \cite{kendall},\cite{kibble} and the use of the polynomials to solve different cases of the \iftex Focker-Planck-Kolmogorov(FPK)\fi \ifblog \href{https://en.wikipedia.org/wiki/Fokker\%E2\%80\%93Planck_equation}{Focker-Planck-Kolmogorov(FPK)}\fi equation, in particular for the \iftex Ornstein–Uhlenbeck \fi \ifblog \href{https://en.wikipedia.org/wiki/Ornstein\%E2\%80\%93Uhlenbeck_process}{Ornstein–Uhlenbeck}\fi process \cite{pavliotis},\cite{risken}  were not covered. Nevertheless we hope that the reader obtained some appreciation for  \iftex\calligra The {Chebyshev-Hermite} Polynomials \normalfont\fi \ifblog the \textit{Chebyshev-Hermite} polynomials\fi .  

\footreferences

\end{document}